\documentclass[10pt]{amsart}

\usepackage{amsthm, amsmath, amssymb}
\usepackage{mathtools} 
\usepackage{mathrsfs}  
\usepackage{bm} 

\usepackage{tikz}
\usepackage{tikz-cd} 
\usepackage{graphicx} 

\usepackage{enumitem} 
\usepackage{aliascnt} 

\usepackage[english]{babel}
\usepackage[utf8]{inputenc} 
\usepackage[T1]{fontenc} 

\usepackage[allcolors=black]{hyperref} 
\usepackage[noadjust]{cite} 
\usepackage[disable]{todonotes} 

	\theoremstyle{plain}
		
		\newtheorem{thm}{Theorem}[section]	
		\newtheorem{lem}[thm]{Lemma}
		\newtheorem{prop}[thm]{Proposition}

		\newtheorem{claim}{Claim}
		
		\newtheorem{thmintro}{Theorem}

	\theoremstyle{definition} 
	
		\newtheorem{defn}[thm]{Definition}
		\newtheorem{eg}[thm]{Example}
		
		\newtheorem{conv}[thm]{Convention}
	
	\theoremstyle{remark}
		\newtheorem{rem}[thm]{Remark}

		\newcommand{\N}{\mathbb N}

		\newcommand{\C}{\mathbb C}
		\newcommand{\R}{\mathbb R}

		\DeclareMathOperator{\SL}{SL}
		
		\DeclareMathOperator{\Hom}{Hom}
		\DeclareMathOperator{\End}{End}

		\DeclareMathOperator{\supp}{supp}

		\DeclareMathOperator{\id}{id}

		\DeclareMathOperator{\Ad}{Ad}

		\DeclareMathOperator{\Ind}{Ind}
		
		\DeclareMathOperator{\Iso}{Iso}
		
		\newcommand{\acton}{\curvearrowright}
		\newcommand{\GG}{G^{(0)}}

\title[Amenable actions and Fell bundles]{Amenability for actions of \'etale groupoids on $C^*$-algebras and Fell bundles}

\author{Julian Kranz}
\address{Mathematisches Institut, Westf\"alische Wilhelms-Universit\"at M\"unster, Einsteinstr. 62, 48149 M\"unster, Germany} 
\email{julian.kranz@uni-muenster.de}
\urladdr{https://www.uni-muenster.de/IVV5WS/WebHop/user/j\_kran05/}

\subjclass[2020]{46L55 (Primary), 22A22 (Secondary).}

\date{\today}

\thanks{This work was funded by the Deutsche Forschungsgemeinschaft (DFG, German Research Foundation) - Project-ID 427320536 - SFB 1442, as well as by Germany's Excellence Strategy EXC 2044 390685587, Mathematics Münster: Dynamics-Geometry-Structure.}

\keywords{$C^*$-algebras, Fell bundles, groupoid actions, amenability, nuclearity}

\begin{document}

\begin{abstract}
We generalize Renault's notion of measurewise amenability to actions of second countable, Hausdorff, \'etale groupoids on separable $C^*$-algebras and show that measurewise amenability characterizes nuclearity of the crossed product whenever the $C^*$-algebra acted on is nuclear. In the more general context of Fell bundles over second countable, Hausdorff, \'etale groupoids, we introduce a version of Exel's approximation property. We prove that the approximation property implies nuclearity of the cross-sectional algebra whenever the unit bundle is nuclear. For Fell bundles associated to groupoid actions, we show that the approximation property implies measurewise amenability of the underlying action. 

\end{abstract}

\maketitle

\tableofcontents

\section{Introduction}
Originally introduced by John von Neumann for discrete groups \cite{Neumann1929}, the notion of amenability has been generalized over several decades to larger and larger classes of dynamical systems, such as group actions on spaces \cite{Zimmer}, locally compact groupoids \cite{ADRenault2000amenableGroupoids}, group actions on operator algebras \cite{anantharaman1979amenablevN,anantharaman1987amenableC,BusEchWil22} and even to Fell bundles over groups \cite{Exel-Crelle,exel2002approximation, abadie2021amenability}.
While the theory of amenable groupoids is by now a well-known and established theory (see the extensive textbooks \cite{ADRenault2000amenableGroupoids, williamsgroupoidbook}), the theory of amenability for group actions on $C^*$-algebras has recently experienced a number of breakthroughs. These include the introduction of the correct definition for actions of locally compact groups \cite{BusEchWil22} and the discovery of several useful characterizations of amenability \cite{abadie2021amenability,BeardenCrann,OzawaSuzuki}. 
Prior to that, approximation properties like the weak containment property or nuclearity of the associated $C^*$-algebras have also been studied for actions of and Fell bundles over \emph{amenable groupoids} \cite{williamsamenability,takeishi,LalondeNuclearity,Lal19}. However, an approach to these phenomena focusing on amenability of the dynamical system instead of amenability of the underlying groupoids does not seem to have been developed so far.

The aim of this work is to push towards a unification of the above mentioned theories by introducing amenability conditions for actions of groupoids on operator algebras. 
More specifically, we introduce notions of 

\begin{enumerate}
	\item \emph{Amenable actions} of groupoids on von Neumann algebras, generalizing both amenable measured groupoids \cite{ADRenault2000amenableGroupoids} and amenable actions of groups on von Neumann algebras \cite{anantharaman1979amenablevN} (see Definition \ref{defn:measurewise amenable}),
	\item \emph{Measurewise amenable actions} of groupoids on $C^*$-algebras, generalizing both measurewise amenable groupoids \cite{ADRenault2000amenableGroupoids} and amenable actions of groups on $C^*$-algebras \cite{anantharaman1987amenableC, BusEchWil22} (see Definition \ref{defn: amenable G C^*}), and
	\item An \emph{approximation property} for Fell bundles over groupoids, generalizing both topological amenability for groupoids \cite{ADRenault2000amenableGroupoids} and Exel's approximation property \cite{Exel-Crelle} (see Definition \ref{defn:AP}).	
\end{enumerate}

In the definitions mentioned above and throughout the paper, we restrict our attention to locally compact, second countable, Hausdorff, \emph{\'etale} groupoids as well as $C^*$-algebras, von Neumann algebras and \emph{saturated} Fell bundles satisfying appropriate separability conditions. Our reason to work in a separable setting is our crucial use of Renault's disintegration theorem \cite{Renault1987Representations} and its generalization to saturated Fell bundles \cite{Muhlydisintegration}. The \'etaleness assumption is not needed for the definitions themselves, but is used in most of our theorems and makes some technicalities more accessible. Although parts of the paper carry over to the non-\'etale case, we keep \'etaleness as a blanket assumption and outline the non-\'etale case in an appendix. We hope that this increases the readability of the paper.

Our main results state that two of the key properties of actions of amenable groupoids on $C^*$-algebras and Fell bundles over amenable groupoids, namely
\begin{enumerate}
	\item Equality of the full and reduced $C^*$-algebras,
	\item Nuclearity of the reduced $C^*$-algebra, given that unit fibers are nuclear,
\end{enumerate}
can still be obtained when relaxing amenability of the underlying groupoid to an amenability condition on the dynamical system (see Theorems \ref{thm:APimpliesWC}, \ref{thm:measurewise amenable implies WC}, \ref{thm:APimpliesnuclearity}, and \ref{thm:Nuclearity for groupoid crossed products} for the precise statements). For actions of groupoids on $C^*$-algebras, we also obtain a converse:

\begin{thmintro}[Theorem \ref{thm:Nuclearity for groupoid crossed products}]
	Let $\alpha\colon G\acton A$ be an action of a second countable, Hausdorff, \'etale groupoid on a separable $C^*$-algebra. Then the reduced crossed product $A\rtimes_r G$ is nuclear if and only if $A$ is nuclear and $\alpha$ is measurewise amenable. 
\end{thmintro}

We also relate the approximation property to measurewise amenability:
\begin{thmintro}[Theorem \ref{thm:AP implies measurewise amenable}]
	Let $\alpha\colon G\acton A$ be an action of a second countable, Hausdorff, \'etale groupoid on a separable $C^*$-algebra. Assume that the associated Fell bundle has the approximation property. Then $\alpha$ is measurewise amenable. 
\end{thmintro}

We conclude this introduction with some questions that we leave unanswered. 
\begin{enumerate}
	\item Let $\alpha\colon G\acton A$ be an action of a second countable, Hausdorff, \'etale groupoid on a separable $C^*$-algebra. Does the associated Fell bundle have the approximation property if and only if $\alpha$ is measurewise amenable?\label{Q:topmeasure}
	\item Let $\mathscr B$ be a separable, saturated Fell bundle over a second countable, Hausdorff, \'etale groupoid $G$. Does nuclearity of $C^*_r(\mathscr B)$ imply that $\mathscr B$ has the approximation property? \label{Q:nuclearity}
	\item Is the approximation property invariant under equivalence of Fell bundles?\label{Q:equivalence}
\end{enumerate}
A positive solution to both Question \eqref{Q:topmeasure} and Question \eqref{Q:equivalence} would yield a positive solution to Question \eqref{Q:nuclearity} since every separable, saturated Fell bundle over a locally compact, Hausdorff, second countable groupoid is equivalent to an action on a separable $C^*$-algebra \cite{WilliamsStabilization}. 
Note also that all the above questions have positive answers for \'etale groupoids acting on their unit space (see \cite[Cor. 6.2.14, Thm. 3.3.7, Thm. 2.2.17]{ADRenault2000amenableGroupoids}) and for Fell bundles over (or actions of) discrete groups (see \cite[Thm. 2.13, Thm. 3.2]{OzawaSuzuki}, \cite[Cor. 4.11]{BusEchWil22}, \cite[Cor. 5.20]{abadie2021amenability}). 

\subsection*{Acknowledgements}
The author would like to thank Alcides Buss, Siegfried Echterhoff, and Diego Mart\'inez for several useful discussions and comments, and the referee for pointing out the reference \cite{Lal19}. The author is especially grateful to Siegfried Echterhoff for continuously encouraging him to push forward this project.

\section{Preliminaries}
\subsection{Notation}
	If $X,Y,Z$ are sets and $p\colon X\to Z,~q\colon Y\to Z$ are maps, we denote by 
		\[X\times_{p,Z,q}Y\coloneqq \{(x,y)\in X\times Y\mid p(x)=q(y)\}\]
	their fiber product. If $p$ or $q$ are clear from the context, we omit them from the notation. If we understand $q\colon Y\to Z$ as a bundle over $Z$, we also call 
		\[p^*Y\coloneqq X\times_{p,Z,q}Y\]
	the \emph{pull-back of $Y$ along $p$} and understand it as a bundle over $X$ via the first coordinate projection. 
	Throughout, all measures on standard Borel spaces are assumed to be $\sigma$-finite and all measures on locally compact, Hausdorff spaces are assumed to be Radon measures. 
	If $A$ is a $C^*$-algebra and $E$ a Hilbert-$A$-module, we denote by $\mathcal L_A(E)$ the $C^*$-algebra of adjointable operators on $E$. If $B$ is another $C^*$-algebra and $\pi\colon A\to \mathcal L_B(F)$ a representation on a Hilbert-$B$-module $F$, we denote by $E\otimes_\pi F$ the balanced tensor product of $E$ and $F$ over $\pi$, see \cite[Ch. 4]{Lance}.
	
\subsection{Hilbert bundles}
For more background on Hilbert bundles, we refer the reader to \cite[Appendix F]{williams07} or to \cite[Ch. IV.8]{takesaki1}. 
\begin{defn}\label{defn:Hilbert bundle}
	Let $X$ be a standard Borel space. A \emph{Hilbert bundle over $X$} is a Borel space $\mathscr H$ together with a Borel map $p\colon\mathscr H\to X$, and for each $x\in X$  a Hilbert space structure on the fiber $\mathscr H_x\coloneqq r^{-1}(x)$, (inducing the given Borel structure), and such that there exists a sequence $(\xi_n)_{n\in \N}$ of Borel sections $X\to \mathscr H$ satisfying the following properties. 
	\begin{enumerate}
		\item The maps $\mathscr H\to \C,\quad \eta\mapsto \langle \eta,\xi_n(p(\eta))\rangle$ are Borel for each $n\in \N$.\label{item:Hilbert1}
		\item The maps $X\to \C,\quad x\mapsto \langle \xi_n(x),\xi_m(x)\rangle$ are Borel for each $n,m\in \N$.\label{item:Hilbert2}
		\item The maps $p$ and $\eta\mapsto \langle \eta,\xi_n(p(\eta))\rangle$ together separate the points in $\mathscr H$. \label{item:Hilbert3}
	\end{enumerate}
\end{defn}

\begin{thm}[{\cite[Prop. F.8]{williams07}}]\label{thm:construct Hilbert bundle}
	Let $X$ be a standard Borel space and $\mathscr H$ a set together with a projection $p\colon \mathscr H\to X$ such that for each $x\in X$, the fiber $\mathscr H_x\coloneqq p^{-1}(x)$ is equipped with the structure of a Hilbert space. Suppose there is a sequence $(\xi_n)_{n\in \N}$ of Borel sections $X\to \mathscr H$ satisfying conditions \eqref{item:Hilbert2} and \eqref{item:Hilbert3} of Definition \ref{defn:Hilbert bundle}. Then there is a unique Borel structure on $\mathscr H$ such that $p$ and $\xi_n$ for all $n\in \N$ are Borel and such that condition \eqref{item:Hilbert1} of Definition \ref{defn:Hilbert bundle} holds as well. 
\end{thm}
We call a sequence $(\xi_n)_{n\in \N}$ as above a \emph{fundamental sequence for $\mathscr H$}. 
If $\mathscr H$ is a Hilbert bundle over a standard Borel space $X$, we equip 
	\[\Hom(\mathscr H)\coloneqq \left\{(x,T,y)\mid x, y\in X, V\in \mathcal L(\mathscr H_x,\mathscr H_y)\right\}\]
with the weakest Borel structure such that the maps 
	\[\Hom(\mathscr H)\to \C,\quad (x,T,y)\mapsto \langle \xi(y),T \eta(x)\rangle\]
are Borel for all Borel sections $\xi,\eta\colon X\to \mathscr H$. We denote by 
	\[\Iso(\mathscr H)\coloneqq \left\{(x,V,y)\mid V \text{ unitary}\right\}\subseteq \Hom(\mathscr H)\]
the \emph{isomorphism groupoid} of $\mathscr H$ and by 
	\[\End(\mathscr H)\coloneqq \left\{(x,T,y)\mid x=y\right\}\subseteq \Hom(\mathscr H)\]
	the \emph{endomorphism bundle} of $\mathscr H$. 

\begin{defn}
	Let $\mu$ be a measure on $X$. We denote by $L^2(X,\mathscr H,\mu)$ (or $L^2(\mathscr H)$ if $X$ and $\mu$ are understood) the space of equivalence classes of measurable sections $\xi\colon X\to \mathscr H$ such that $x\mapsto \|\xi(x)\|^2$ is $\mu$-integrable, with sections declared equivalent if they agree $\mu$-almost everywhere. We equip $L^2(X,\mathscr H,\mu)$ with the inner product 
		\[\langle \xi,\eta\rangle \coloneqq \int_X\langle \xi(x),\eta(x)\rangle d\mu(x),\quad \xi,\eta\in L^2(X,\mathscr H,\mu).\]
\end{defn}

Let $(X,\mathscr H,\mu)$ be as above. A \emph{measurable field of operators} on $\mathscr H$ is a Borel section 
	\[X\to \End(\mathscr H),\quad x\mapsto T_x\in \mathcal L(\mathscr H_x).\]
We call $x\mapsto T_x$ \emph{essentially bounded} if $x\mapsto \|T_x\|$ is bounded on a $\mu$-conull set. In this case, we can define an operator $T=\int_XT_x d\mu(x)$ on $L^2(X,\mathscr H,\mu)$ by $T\xi(x)\coloneqq T_x(\xi(x))$ for $\xi\in L^2(X,\mathscr H,\mu)$ and $x\in X$. 

\begin{defn}
	An operator $T\in \mathcal L(L^2(X,\mathscr H,\mu))$ is called \emph{decomposable}, if it is of the form $T=\int_XT_xd\mu(x)$ for an essentially bounded measurable field of operators $x\mapsto T_x$. A decomposable operator $T$ is called \emph{diagonal}, if $T_x$ can be chosen to be a multiple of the identity for every $x\in X$.
\end{defn}

\begin{thm}[{\cite[Lem. F.19, Thm. F.20]{williams07}}]\label{Thm unique decomposition}
	An operator $T\in \mathcal L(L^2(X,\mathscr H,\mu))$ is decomposable if and only if it commutes with the diagonal operators. If $T=\int_X T_xd\mu(x)$ and $T=\int_X T'_xd\mu(x)$ are two decompositions of $T$, then we have $T_x=T'_x$ for $\mu$-almost every $x\in X$. 
\end{thm}

Let $A$ be a separable $C^*$-algebra and let $(X,\mathscr H,\mu)$ as above. We call a family $\{\pi_x\colon A\to \mathcal L(\mathscr H_x)\}_{x\in X}$ of representations a \emph{Borel family}, if $X\ni x\mapsto \pi_x(a)\in \mathcal L(\mathscr H_x)$ is a measurable field of operators for every $a\in A$. 

\begin{prop}[{\cite[Prop. F.25]{williams07}}]\label{direct integral representation}
	Let $\pi\colon A\to \mathcal L(L^2(X,\mathscr H,\mu))$ be a representation such that $\pi(A)$ commutes with the diagonal operators. Then there is a Borel family of representations $\{\pi_x\colon A\to \mathcal L(\mathscr H_x)\}_{x\in X}$ satisfying
		\[\pi(a)=\int_X\pi_x(a)d\mu(x),\quad \forall a\in A.\]
\end{prop} 

\subsection{\'Etale groupoids}
For more background on (\'etale) groupoids, we refer the reader to the recent textbook \cite{williamsgroupoidbook}. 

A groupoid $G$ is a small category all of whose morphisms are invertible. We denote by $G$ the set of all morphisms and call the set $G^{(0)}\subseteq G$ of identity morphisms the \emph{unit space}. We denote by $r,s\colon G\to G^{(0)}$ the range and source maps, by $\cdot\colon G^{(2)}\coloneqq G\times_{s,G^{(0)},r}G\to G$ the composition, and by $G\to G,\quad g\mapsto g^{-1}$ the inversion. 

A \emph{topological groupoid} is a groupoid $G$ equipped with a topology such that all the maps mentioned above are continuous. A locally compact topological groupoid $G$ is called \emph{\'etale} if $r\colon G\to G^{(0)}$ and $s\colon G\to G^{(0)}$ are local homeomorphisms. A subset $U\subseteq G$ is called a \emph{bisection} if $r|_U$ and $s|_U$ are homeomorphisms onto their images. Note that a locally compact groupoid is \'etale if and only if its topology has a basis of open bisections. If $G$ is \'etale, then the \emph{fibers} $G^x\coloneqq r^{-1}(x)$ and $G_x\coloneqq s^{-1}(x)$ are discrete for every $x\in G^{(0)}$. 

A subset $U\subseteq G^{(0)}$ is \emph{invariant} if we have $r(g)\in U\Leftrightarrow s(g)\in U$ for all $g\in G$. If $U$ is invariant, we write $G|_U\coloneqq r^{-1}(U)=s^{-1}(U)$.

\begin{conv}
	From now on, all groupoids are assumed to be Hausdorff, \'etale, and second countable. 
\end{conv}

\begin{defn}\label{defn:G-space}
	Let $G$ be a groupoid. A \emph{$G$-space} is a locally compact Hausdorff space $X$ together with a continuous map $p\colon X\to G^{(0)}$ and a continuous map 
		\[\alpha\colon G\times_{s,G^{(0)},p}X\to G\times_{r,G^{(0)},p}X, \quad (g,x)\mapsto (g,\alpha_g(x))\]
	satisfying $\alpha_g\circ \alpha_h=\alpha_{gh}$ for all $(g,h)\in G^{(2)}$. 
	The associated transformation groupoid is the groupoid $X\rtimes G\coloneqq X\times_{p,G^{(0)},r}G$ with unit space $X$ and range, source and multiplication maps given by
	\[r(x,g)\coloneqq x,\quad s(x,g)\coloneqq \alpha_{g^{-1}}(x),\quad (x,g)(y,h)\coloneqq (x,gh),\]	
	 for $((x,g),(y,h))\in (X\rtimes G)^{(2)}$.
\end{defn}

\begin{defn}\label{defn:quasi-invariant}
	Let $G$ be a groupoid and $\mu$ a measure on $G^{(0)}$. We define a measure $\mu\circ \lambda$ on $G$ by 
		\[\mu \circ \lambda(f)=\int_{G^{(0)}}\sum_{g\in G^x}f(g)d\mu(x),\quad f\in C_c(G).\]
	We call $\mu$ \emph{quasi-invariant}, if $\mu \circ \lambda$ is equivalent to the measure $(\mu \circ \lambda)^{-1}$ obtained by pulling back along the inversion map $g\mapsto g^{-1}$. In that case, the Radon--Nikodym derivative $\Delta\coloneqq \frac{d(\mu \circ \lambda)}{d(\mu \circ \lambda)^{-1}}$ is called the \emph{modular function} of $\mu$. 
\end{defn} 

\begin{lem}[{\cite[Lem. 7.6]{williams07}}]\label{lem:hom-almost-everywhere}
	The modular function $\Delta\colon G\to (0,\infty)$ can be chosen to be a Borel homomorphism. 
\end{lem}

\begin{rem}
	The notation $\mu\circ \lambda$ above alludes to the more general definition when $G$ is a locally compact groupoid with a Haar system $\lambda$. In our situation, $\lambda$ will always be the counting measure. 
\end{rem}

\subsection{Fell bundles}
Fell bundles over \'etale groupoids were introduced in \cite{kumjan1998Fell}. For more details, we refer to \cite{Muhlydisintegration}. For general background on Banach bundles, see \cite{felldoran88}. 

\begin{defn}
	An \emph{upper semicontinuous Banach bundle} over a locally compact Hausdorff space $X$ is a topological space $\mathscr B$ together with a continuous map ${p\colon \mathscr B\to X}$ and for each $x\in X$ the structure of a Banach space on the fiber $\mathscr B_x\coloneqq p^{-1}(x)$ (inducing the given topology) such that 
	\begin{enumerate}
		\item The addition $+\colon \mathscr B\times_X\mathscr B\to \mathscr B$ is continuous.
		\item The scalar multiplication $\cdot\colon\C\times \mathscr B\to \mathscr B$ is continuous.
		\item For each continuous section $f\colon X\to \mathscr B$ of $p$, the function 
			\[X\ni x\mapsto \|f(x)\|\in [0,\infty)\]
		is upper semicontinuous. 
	\end{enumerate}
\end{defn}

\begin{defn}\label{defn:C*bundle}
	An upper semicontinuous $C^*$-bundle over a locally compact Hausdorff space $X$ is an upper continuous Banach bundle $p\colon \mathscr A\to X$ where each fiber $\mathscr A_x$ is equipped with the structure of a $C^*$-algebra (inducing the given norm) such that the multiplication $\cdot \colon \mathscr A\times_X\mathscr A\to \mathscr A$ and the involution $*\colon \mathscr A\to \mathscr A$ are continuous. 
\end{defn}

\begin{rem}\label{rem:C_0(X)-alg}
For an upper semicontinuous $C^*$-bundle $\mathscr A$ over $X$, we denote by $C_0(X,\mathscr A)$ (or $C_0(\mathscr A)$ if $X$ is understood) the $C^*$-algebra of continuous sections $X\to \mathscr A$ vanishing at infinity, equipped with the supremum norm. There is a canonical non-degenerate $*$-homomorphism $C_0(X)\to ZM(C_0(\mathscr A))$. 

More abstractly, we call a $C^*$-algebra $A$ equipped with a non-degenerate $*$-homomorphism $C_0(X)\to ZM(A)$ a \emph{$C_0(X)$-algebra}. We write 
	\[{A_x\coloneqq A/(C_0(X\setminus \{x\})A)}\]
for $x\in X$ and can consider every element $a\in A$ as a map 
	\begin{equation}\label{eq:continuous sections}
		X\to \mathscr A\coloneqq \bigsqcup_{x\in X}A_x,\quad x\mapsto a+C_0(X\setminus \{x\})A.
	\end{equation}
There is a unique topology on $\mathscr A$ turning it into an upper semicontinuous $C^*$-bundle and making the maps defined in \eqref{eq:continuous sections} continuous {\cite[Thm.~C.25]{williams07}}.
In fact, the functor $\mathscr A\mapsto C_0(\mathscr A)$ defines an equivalence from the category of upper semicontinuous $C^*$-bundles over $X$ with continuous fiberwise $*$-homomorphisms to the category of $C_0(X)$-algebras with $C_0(X)$-linear $*$-homomorphisms. 
\end{rem}

\begin{defn}[{\cite{kumjan1998Fell}}]\label{defn:Fell bundle}
	A \emph{Fell bundle} over $G$ is an upper semicontinuous Banach bundle $p\colon \mathscr B\to G$ together with a continuous multiplication map 
		\[\cdot\colon \mathscr B^{(2)}\coloneqq \mathscr B\times_{s\circ p,G,r\circ p}\mathscr B  \to \mathscr B\] and a continuous involution map $*\colon \mathscr B\to \mathscr B$ satisfying the identities
		\begin{enumerate}
			\item $p(bc)=p(b)p(c)$ \label{item:first Fell condition}
			\item $p(b^*)=p(b)^{-1}$
			\item $(b^*)^*=b$
			\item $(bc)^*=c^*b^*$
			\item $\|b^*b\|=\|b\|^2$
			\item $\|b\|\|c\|\leq \|bc\|$
			\item $b^*b\geq 0$
		\end{enumerate}
	for all $(b,c)\in \mathscr B^{(2)}$, as well as the following properties.
		\begin{enumerate}
		\setcounter{enumi}{7}
			\item The multiplication is bilinear and associative.
			\item The involution is conjugate linear.
		\end{enumerate}
	In addition, we say that
		\begin{enumerate}
		\setcounter{enumi}{9}
			\item $\mathscr B$ is \emph{saturated}, if for each $g\in G$, we have $\overline{\operatorname{span} \mathscr B_{g}^*\mathscr B_g} =\mathscr B_{s(g)}$.\label{item:saturated}
			\item $\mathscr B$ is \emph{separable}, if the Banach space $C_0(G,\mathscr B)$ of $C_0$-sections of $p\colon \mathscr B\to G$ is separable. \label{item:separable}
		\end{enumerate}
\end{defn}
We abbreviate the notation by saying that \emph{$\mathscr B$ is a Fell bundle over $G$}. Note that the restriction $\mathscr B^{(0)}$ of $\mathscr B$ to $G^{(0)}$ is an upper semicontinuous $C^*$-bundle over $G^{(0)}$.

From now on, we keep conditions \eqref{item:saturated} and \eqref{item:separable} in Definition \ref{defn:Fell bundle} as blanket assumptions without mentioning them:

\begin{conv}
	Throughout this paper, all Fell bundles are assumed to be separable and saturated. 
\end{conv}

\begin{rem}\label{rem:imprimitivity bimodule}
Condition \eqref{item:saturated} ensures that for every $g\in G$, $\mathscr B_g$ is a $\mathscr B_{r(g)}$-$\mathscr B_{s(g)}$-imprimitivity bimodule with respect to the obvious module structure and the inner products given by 
	\[\langle a,b\rangle_{\mathscr B_{s(g)}}\coloneqq a^*b,\quad {}_{\mathscr B_{r(g)}}\langle a,b\rangle \coloneqq ab^*,\quad a,b\in \mathscr B_g.\]
\end{rem}	

\begin{defn}\label{defn:C-dyn-sys}
	A \emph{$C^*$-dynamical system} $(G,\mathscr A,\alpha)$ is given by a groupoid $G$, a separable upper semicontinuous $C^*$-bundle $\mathscr A$ over $G^{(0)}$, and a continuous isomorphism 
		\[\alpha\colon \mathscr A\times_{G^{(0)},s}G\to \mathscr A\times_{G^{(0)},r}G\]
	of upper semicontinuous $C^*$-bundles over $G$, satisfying 
		$\alpha_{gh}=\alpha_g\circ \alpha_h$ for all $(g,h)\in G^{(2)}$. 
\end{defn}

Writing $A=C_0(\mathscr A)$, we may alternatively say that $(A,\alpha)$ is a \emph{$G$-$C^*$-algebra} or that $\alpha\colon G\acton A$ is \emph{an action}. Note that if $A$ commutative, then the Gelfand Spectrum $X$ of $A$ is a $G$-space in the sense of Definition \ref{defn:G-space}.

\begin{defn}\label{defn:semidirect product}
	Let $(G,\mathscr A,\alpha)$ be a $C^*$-dynamical system. The associated \emph{semi-direct product bundle} is the Fell bundle $\mathscr A\rtimes_\alpha G\coloneqq \mathscr A\times_{G^{(0)},r}G$, equipped with operations $(a,g)\cdot (b,h)\coloneqq  (a \alpha_g(b),gh)$ and $(a,g)^*\coloneqq (\alpha_{g^{-1}}(a^*),g^{-1})$ for $(g,h)\in G^{(2)}, a\in \mathscr A_{r(g)}$, and $b\in \mathscr A_{r(h)}$. 
\end{defn}

If $\mathscr B$ is a Fell bundle over $G$, we denote by $C_c(G,\mathscr B)$ (or $C_c(\mathscr B)$ if $G$ is understood) the space of compactly supported continuous sections of $p\colon \mathscr B\to G$. We equip $C_c(\mathscr B)$ with a multiplication and involution given by 
	\[f_1*f_2(g)\coloneqq \sum_{h\in G^{r(g)}}f_1(h)f_2(h^{-1}g),\quad f_1^*(g)\coloneqq f_1(g^{-1})^*\]
for $f_1,f_2\in C_c(\mathscr B)$ and $g\in G$. 

\begin{defn}[\cite{Yamagami}]\label{defn:max-cross-sec}
	\begin{sloppypar}
	The \emph{maximal cross-sectional algebra} $C^*(\mathscr B)\coloneqq C^*(G,\mathscr B)$ of $\mathscr B$ is the enveloping $C^*$-algebra of $C_c(\mathscr B)$.
	\end{sloppypar} 
\end{defn}
If $\mathscr B=\mathscr A\rtimes_\alpha G$ is a semi-direct product bundle as in Definition \ref{defn:semidirect product}, we also write $C_0(\mathscr A)\rtimes G\coloneqq C^*(\mathscr A\rtimes_\alpha G)$. 

\begin{rem}
Definition \ref{defn:max-cross-sec} above is equivalent to the definition in \cite[Section 1]{Muhlydisintegration} for the following reason. 
As observed in \cite{BussExel12}, the same argument as in \cite[Prop. 3.14]{ExelCombinatorial} shows that $\|f\|\leq \|f\|_\infty$ for every element $f\in C_c(\mathscr B)$ whose support is contained in an open bisection. 
Together with a partition of unity argument this implies that every representation $\pi$ of $C_c(\mathscr B)$ on a Hilbert space is continuous with respect to the inductive limit topology and therefore a representation in the sense of \cite[Definition 4.7]{Muhlydisintegration}. It then follows from \cite[Theorem 4.13]{Muhlydisintegration} that $\pi$ is \emph{$I$-norm-decreasing} and therefore extends to the maximal cross-sectional algebra as defined in \cite{Muhlydisintegration}.
\end{rem}

\begin{defn}[\cite{kumjan1998Fell}]
	Let $\mathscr B$ be a Fell bundle over $G$. We denote by $L^2(G,\mathscr B)$ (or $L^2(\mathscr B)$ if $G$ is understood) the completion of $C_c(\mathscr B)$ with respect to the $C_c(\mathcal B^{(0)})$-valued inner product given by 
		\[\langle \xi,\eta\rangle (x)\coloneqq \sum_{g\in G_{x}}\xi(g)^*\eta(g),\quad \xi,\eta\in C_c(\mathscr B),x\in G^{(0)}.\]
	The \emph{regular representation} of $\mathscr B$ is the representation 
		\[\Lambda\colon C_c(\mathscr B)\to \mathcal L_{C_0(\mathscr B^{(0)})}(L^2(\mathscr B))\]
		defined by
		\[\Lambda(f)\xi(g)=\sum_{h\in G^{r(g)}} f(h)\xi(h^{-1}g),~f\in C_c(\mathscr B),\xi\in L^2(\mathscr B),g\in G.\]
	We call $C^*_r(\mathcal B)\coloneqq \overline{\Lambda(C_c(\mathscr B))}^{\|\cdot\|}$ the \emph{reduced cross-sectional algebra of $\mathscr B$}.
\end{defn}

If $\mathscr B=\mathscr A\rtimes_\alpha G$ is a semi-direct product bundle as in Definition \ref{defn:semidirect product}, we also write $C_0(\mathscr A)\rtimes_r G\coloneqq C^*_r(\mathscr A\rtimes_\alpha G)$.  

Like upper semicontinuous $C^*$-bundles, Fell bundles can be constructed by prescribing their algebraic structure and a set of continuous sections. 
 This idea goes back to \cite[II.13.18]{felldoran88} and is spelled out in \cite[Prop. 2.5]{takeishi} for \emph{continuous} Fell bundles.

\begin{defn}\label{defn:algebraic-Fell-bundle}
	An \emph{algebraic Fell bundle} over $G$ is a set $\mathscr B$ together with a projection map $p\colon \mathscr B\to G$, a Banach space structure on $\mathscr B_g\coloneqq p^{-1}(g)$ for each $g\in G$, a multiplication map $\cdot\colon \mathscr B\times \mathscr B\to \mathscr B$, and an involution map $*\colon \mathscr B\to \mathscr B$ satisfying the conditions \eqref{item:first Fell condition}-\eqref{item:saturated} of Definition \ref{defn:Fell bundle}.\footnote{In other words, an algebraic Fell bundle over $G$ is a (possibly non-separable) Fell bundle over $G$, considered as a discrete groupoid.} 
\end{defn}

\begin{prop}[cp. {\cite[Prop. 2.5]{takeishi}}]\label{prop:algebraic Fell bundles}
	Let $p\colon \mathscr B\to G$ be an algebraic Fell bundle over $G$. Let $A_0$ be a countably generated $*$-algebra of compactly supported sections of $\mathscr B$ satisfying
	\begin{enumerate}
		\item For every $g\in G$, the set $\{f(g)\colon f\in A_0\}$ is dense in $\mathscr B_g$. \label{item:fiberwise dense}
		\item For every $f\in A_0$, the map $g\mapsto \|f(g)\|$ is upper semicontinuous. \label{item:upper semicontinuous}
	\end{enumerate}
	Then there is a unique topology on $\mathscr B$ turning it into a Fell bundle such that all sections in $A_0$ are continuous. 
	If furthermore $A_0$ is closed under pointwise multiplication with functions in $C_c(G)$, then $A_0$ is dense in $C_c(\mathscr B)$ in the inductive limit topology. In particular, $A_0$ is norm dense in $C^*(\mathscr B)$. 
\end{prop}
\begin{proof}
		By the same proof as in \cite[Thm. C25]{williams07}, or by the combination of \cite[Prop. 3.6]{HofmannBundles} and the proof of \cite[Prop. C20]{williams07}, there is a unique topology on $\mathscr B$ turning it into an upper semicontinuous Banach bundle over $G$ such that all the sections in $A_0$ are continuous. By the same argument as in \cite[Prop. 2.5]{takeishi}, the multiplication and involution on $\mathscr B$ are continuous. Note that $\mathscr B$ is separable since $A_0$ is countably generated. 
		
		We prove that $A_0$ is dense in $C_c(\mathscr B)$ with respect to the inductive limit topology. Fix a function $f\in C_c(\mathscr B)$, a relatively compact open neighbourhood $U$ of $\supp f$ and $\varepsilon>0$. For each $x\in \supp f$, pick a function $f_x\in A_0$ with $\|f(x)-f_x(x)\|<\frac \varepsilon 3$ and an open neighbourhood $U_x\subseteq U$ of $x$ such that for each $y\in U_x$, we have $\|f(x)-f(y)\|<\frac \varepsilon 3$ and $\|f_x(x)-f_x(y)\|< \frac \varepsilon 3$. By compactness, there is a finite set $\{x_1,\dotsc,x_n\}\subseteq \supp f$ such that $\supp f\subseteq U_{x_1}\cup\dotsb\cup U_{x_n}$. Now pick continuous functions $\varphi_1,\dotsc,\varphi_n\colon G\to [0,1]$ such that $\supp \varphi_i\subseteq U_{x_i}$ and such that $\sum_{i=1}^n \varphi_i(x)=1$ for all $x\in \supp f$. Then $f'\coloneqq \sum_i \varphi_if_{x_i}$ belongs to $A_0$ by assumption. We have $\|f-f'\|_\infty<\varepsilon$ and $\supp f'\subseteq \overline{U}$ where the compact set $\overline{U}$ only depends on $f$. Thus, $A_0$ is dense in $C_c(\mathscr B)$ in the inductive limit topology. 
\end{proof}

\begin{rem}
	If one strengthens density to equality in condition \eqref{item:fiberwise dense} of the above proposition, then one can even drop the assumption that $G$ is Hausdorff (see \cite[Prop. 2.4]{BussExel12}). If $A_0$ is not closed under multiplication by $G$, it may fail to be dense in $C_c(\mathscr B)$. This happens for instance if $\mathscr B=\C\times G$ is the trivial bundle over a finite group $G$ and if $A_0$ is generated by the projection $p\coloneqq \frac 1 {|G|}\sum_{g\in G}(1,g)$.
\end{rem}

We wrap up the preliminaries about Fell bundles by discussing their representation theory. This is the main point where our separability assumptions are used. 

\begin{defn}[{\cite{Muhlydisintegration}}]\label{defn:Borel*functor}
	Given a Fell bundle $\mathscr B$ over a groupoid $G$ and a Hilbert bundle $\mathscr H$ over $G^{(0)}$, a \emph{$*$-functor} is a map 
		\[\pi\colon \mathscr B\to \Hom(\mathscr H),\quad \mathscr B_g\ni b\mapsto \pi_g(b)\in \mathcal L(\mathscr H_{s(g)},\mathscr H_{r(g)})\]
	such that 
		\begin{enumerate}
			\item For each $g\in G$, the map $\pi_g\colon \mathscr B_g\to \mathcal L(\mathscr H_{s(g)},\mathscr H_{r(g)})$ is linear.
			\item $\pi_{gh}(ab)=\pi_g(a)\pi_h(b)$ for all $(g,h)\in G^{(2)}, a\in \mathscr B_g$ and $b\in \mathscr B_h$.
			\item $\pi_g(a)^*=\pi_{g^{-1}}(a^*)$ for all $g\in G$ and $a\in \mathscr B_g$.
		\end{enumerate}
	We call $\pi$ \emph{Borel}, if for each $f\in C_c(\mathscr B)$ and any two Borel sections $\xi,\eta\colon G^{(0)}\to \mathscr H$, the map
		\[G\ni g\mapsto \left\langle \xi(r(g)),\pi_g(f(g))\eta(s(g))\right\rangle \in \C\]
	is Borel.
\end{defn}

\begin{defn}[{cp. \cite{Muhlydisintegration}}]
	Let $\mathscr B$ be a Fell bundle over $G$. A \emph{strict representation} $(\mu,\mathscr H,\pi)$ of $\mathscr B$ is given by a quasi-invariant measure $\mu$ on $G^{(0)}$, a Hilbert bundle $\mathscr H$ over $G^{(0)}$ and a Borel $*$-functor $\pi\colon \mathscr B\to \Hom(\mathscr H)$, such that $\pi_x\colon \mathscr B_x\to \mathcal L(\mathscr H_x)$ is non-degenerate for $\mu$-almost every $x\in G^{(0)}$. 
\end{defn}
In contrast to our definition, the strict representations in  \cite{Muhlydisintegration} are allowed to be degenerate. 

\begin{prop}[{\cite[Prop. 4.10]{Muhlydisintegration}}]
	Let $(\mu,\mathscr H,\pi)$ be a strict representation of a Fell bundle $\mathscr B$ over $G$ and denote by $\Delta\colon G\to (0,\infty)$ the modular function of $\mu$. Then the map $L_\pi\colon C_c(\mathscr B)\to \mathcal L(L^2(\mathscr H))$ given by 
		\[L_\pi(f)\xi(x)=\sum_{g\in G^x}\pi_g(f(g))\xi(s(g))\Delta^{-\frac 1 2}(g)\]
		for $f\in C_c(\mathscr B),\xi\in L^2(\mathscr H)$, and $x\in G^{(0)}$
	is a non-degenerate representation.  
\end{prop}

The representation $L_\pi$ is called the \emph{integrated form} of $(\mu,\mathscr H,\pi)$. By definition, it extends to a representation of $C^*(\mathscr B)$. 
The following generalization of Renault's disintegration theorem shows that every non-degenerate representation of $C^*(\mathscr B)$ arises in this way:

\begin{thm}[{\cite[Thm. 4.13]{Muhlydisintegration}}]\label{thm:Fell-disintegration}
	Let $\mathscr B$ be a Fell bundle over $G$. Then every non-degenerate representation of $C^*(\mathscr B)$ on a separable Hilbert space is equivalent to the integrated form of a strict representation. 
\end{thm}

For semi-direct product bundles associated to $C^*$-dynamical systems we can say a bit more:

\begin{defn}
	 A \emph{covariant representation} $(\mu,\mathscr H,\pi,U)$ of a $C^*$-dynamical system $(G,\mathscr A,\alpha)$ consists of 
		\begin{enumerate}
			\item a quasi-invariant measure $\mu$ on $G^{(0)}$,
			\item a Hilbert bundle $\mathscr H$ over $G^{(0)}$,
			\item a non-degenerate $C_0(G^{(0)})$-linear representation $\pi\colon C_0(\mathscr A)\to \mathcal L(L^2(\mathscr H))$,
			\item a Borel homomorphism $U\colon G\to \Iso(\mathscr H),\quad g\mapsto U_g\in \mathcal U(\mathscr H_{s(g)},\mathscr H_{r(g)})$,
		\end{enumerate}			
	such that for $\mu\circ \lambda$-almost all $g\in G$, we have 
		\begin{equation}\label{eq:covariance condition}
			\pi_{r(g)}(\alpha_g(a))=U_g \pi_{s(g)}(a)U_g^*,\quad \forall a\in \mathscr A_{s(g)},
		\end{equation}
	where $\left(\pi_x\colon \mathscr A_x\to \mathcal L(\mathscr H_x)\right)_{x\in G^{(0)}}$ is any Borel family of representations satisfying $\pi=\int_{G^{(0)}}\pi_xd\mu(x)$ as in Proposition \ref{direct integral representation}.
\end{defn}

Note that set of all $g\in G$ satisfying \eqref{eq:covariance condition} always contains a set of the form $G|_{\mathcal U}$ for a $G$-invariant $\mu$-conull subset $\mathcal U\subseteq G^{(0)}$ (see \cite[Rem 7.10.]{MuhlyWilliams2008crossedproducts}). 

\begin{defn}
The \emph{integrated form} of a covariant representation $(\mu,\mathscr H,\pi,U)$ as above is the representation
	\[ 
		\pi\rtimes U\colon C_c(G,r^*\mathscr A)\to \mathcal L(L^2(\mathscr H))
	\]
	given by
		\[\pi\rtimes U(f)\xi(x)\coloneqq \sum_{g\in G^x}\pi(f(g))U_g(\xi(s(g)))\Delta^{-\frac 1 2}(g),\]
	for $f\in C_c(G,r^*\mathscr A),~\xi\in L^2(\mathscr H)$ and $x\in \GG$, where $\Delta$ denotes the modular function of $\mu$.
\end{defn}

By definition, $\pi\rtimes U$ as above extends to a representation of $C_0(\mathscr A)\rtimes G$. We have the following converse, known as \emph{Renault's disintegration theorem}.
\begin{thm}[{\cite[Lem. 4.6]{Renault1987Representations}, see also \cite[Thm. 7.12]{MuhlyWilliams2008crossedproducts}}]\label{thm:covariant-disintegration}
	Let $(G,\mathscr A,\alpha)$ be a $C^*$-dynamical system. Then every non-degenerate representation of $C_0(\mathscr A)\rtimes G$ on a separable Hilbert space is equivalent to the integrated form of a covariant representation.
\end{thm}

\subsection{$G$-von Neumann algebras}

We recall the definition and basic properties of $G$-von Neumann algebras (or $W^*$-dynamical systems) as introduced in \cite{yamanouchiduality} (see also \cite{OtyRamsay2006GvNA}). 
For more background on bundles of von Neumann algebras, we refer the reader to \cite[Ch. IV.8]{takesaki1}.

\begin{defn}
	Let $X$ be a standard Borel space. A \emph{bundle of von Neumann algebras over $X$} is a Borel subbundle $\mathscr M\subseteq \End(\mathscr H)$ for some Hilbert bundle $\mathscr H$ over $X$, such that $\mathscr M_x\subseteq \mathcal L(\mathscr H_x)$ is a von Neumann algebra for every $x\in X$. 
	For a measure $\mu$ on $X$, we denote by 
		\[L^\infty(X,\mathscr M,\mu)\subseteq \mathcal L(L^2(X,\mathscr H,\mu))\]
	(or $L^\infty(\mathscr M)$ if $X$ and $\mu$ are understood)
	the von Neumann algebra generated by all $\mu$-essentially bounded Borel sections $X\to \mathscr M$. 
\end{defn}

Note that $L^\infty(\mathscr M)$ consists of decomposable operators. Conversely we have:

\begin{thm}[{\cite[Ch. IV, Thm. 8.21]{takesaki1}}]
	Let $X$ be a standard Borel space equipped with a measure $\mu$. Let $H$ be a separable Hilbert space and $M\subseteq \mathcal L(H)$ a von Neumann algebra together with a unital inclusion $\iota\colon L^\infty(X,\mu)\hookrightarrow Z(M)$. Then there is a Hilbert bundle $\mathscr H$ over $X$, a bundle of von Neumann algebras $\mathscr M\subseteq \End(\mathscr H)$ and a unitary $U\colon L^2(X,\mathscr H,\mu)\xrightarrow{\cong}H$ such that $\Ad(U)$ restricts to an isomorphism $L^\infty(X,\mathscr M,\mu)\xrightarrow{\cong}M$ with $\Ad(U)|_{L^\infty(X,\mu)}=\iota$.
\end{thm}

For a von Neumann algebra $M$, we denote by $M_*^+$ the space of normal positive linear functionals on $M$. 

\begin{lem}[{\cite[Ch. IV, Prop. 8.34]{takesaki1}}]\label{lem:normal states on measured fields}
	Let $X$ be a standard Borel space, $\mu$ a measure on $X$ and $\mathscr M$ a bundle of von Neumann algebras over $X$. Let $X\ni x\mapsto \omega_x\in (\mathscr M_x)_*^+$ be a map such that for each $m\in L^\infty(\mathscr M)$, the map $X\ni x\mapsto \omega_x(m(x))\in \C$ is $\mu$-integrable. Then the map 
		\[\omega\coloneqq \int_X\omega_xd\mu(x)\colon L^\infty(\mathscr M)\to \C,\quad m\mapsto \int_X\omega_x(m(x))d\mu(x)\]
	defines a normal positive linear functional on $L^\infty(\mathscr M)$. Every normal positive linear functional on $L^\infty(\mathscr M)$ is of this form. If $\omega=\int_X\omega_x'd\mu(x)$ is another presentation of $\omega$ as above, then we have $\omega_x'=\omega_x$ for $\mu$-almost every $x\in X$. 
\end{lem}

\begin{defn}[cp. \cite{yamanouchiduality}]\label{defn:W*dyn}
	A $W^*$-dynamical system $(G,\mathscr M,\alpha,\mu)$ consists of a groupoid $G$, a bundle of von Neumann algebras $\mathscr M$ over $G^{(0)}$, a quasi-invariant measure $\mu$ on $G^{(0)}$ and a collection $(\alpha_g\colon \mathscr M_{s(g)}\xrightarrow{\cong}\mathscr M_{r(g)})_{g\in G}$ of $*$-isomorphisms satisfying $\alpha_{gh}=\alpha_g\circ \alpha_h$ for all $(g,h)\in G^{(2)}$ and such that for each $m\in L^\infty(\mathscr M)$ and each $\omega=\int_{G^{(0)}}\omega_xd\mu(x)\in L^\infty(\mathscr M)_*^+$, the map 
		\begin{equation}\label{eq:measurable action}
			G\ni g\mapsto \omega_{r(g)}(\alpha_g(m(s(g))))\in \C
		\end{equation}
	is $\mu\circ\lambda$-measurable. 
\end{defn}

Writing $M\coloneqq L^\infty(\mathscr M)$, we may also say that $(M,\alpha,\mu)$ \emph{is a $G$-von Neumann algebra} or that $\alpha\colon G\acton (M,\mu)$ \emph{is an action}. 

\begin{eg}\label{eg:enveloping G von Neumann algebra}
	Let $(G,\mathscr A,\alpha)$ be a $C^*$-dynamical system together with a covariant representation $(\mu,\mathscr H,\pi,U)$. Then there is a $W^*$-dynamical system $(G,\pi(\mathscr A)'',\alpha'',\mu)$ defined by $\pi(\mathscr A)''\coloneqq \bigsqcup_{x\in G^{(0)}}\pi(\mathscr A_x)''\subseteq \End(\mathscr H)$ and 
		\[\alpha''_g(T)\coloneqq U_gTU_g^*,\quad g\in G,T\in \pi(\mathscr A_{s(g)})''.\]
	We have $L^\infty(\pi(\mathscr A)'')=\pi(C_0(\mathscr A))''$. 
\end{eg}
\begin{proof}
	The equality $L^\infty(\pi(\mathscr A)'')=\pi(C_0(\mathscr A))''$ follows from \cite[Ch. IV, Thm. 8.18]{takesaki1}.
	We prove that the map \eqref{eq:measurable action} is measurable for all $\omega\in L^\infty(\pi(\mathscr A)'')_*^+$ and $m\in L^\infty(\pi(\mathscr A)'')=\pi(C_0(\mathscr A))''$.
	Pick a bounded sequence $(a_n)_{n\in \N}$ in $C_0(\mathscr A)$ such that $(\pi(a_n))_{n\in \N}$ converges ultraweakly to $m$. Since $\omega$ is normal, we can pick a sequence $(\xi_i)_{i\in \N}$ in $L^2(G^{(0)},\mathscr H,\mu)$ satisfying $\omega(x)=\sum_{i=0}^\infty \langle \xi_i,x \xi_i\rangle$ for all $x\in L^\infty(\pi(\mathscr A)'')$. We have 
			\[\omega_{r(g)}(\alpha''_g(m(s(g))))=\lim_{n,k\to \infty} \sum_{i=1}^k\left\langle \xi_i(r(g)),U_g\pi_{s(g)}(a_n(s(g)))U_g^*\xi_i(r(g))\right\rangle\]
		for $\mu\circ \lambda$-almost all $g\in G$. As a countable limit of measurable functions in $g\in G$, the above term is a measurable function in $g\in G$. 
\end{proof}

\begin{rem}[cp. \cite{yamanouchiduality,OtyRamsay2006GvNA}]\label{rem:standard form}
	Let $(G,\mathscr M,\alpha,\mu)$ be a $W^*$-dynamical system. Then using the Haagerup standard form \cite{Haagerup1975standard}, one can always find a canonical Hilbert bundle $\mathscr H$, a representation $\mathscr M\subseteq \End(\mathscr H)$ and a unitary representation $U\colon G\to \Iso(\mathscr H)$ satisfying $\alpha_g(m)=U_g m U_g^*$ for all $g\in G$ and $m\in \mathscr M_{s(g)}$. 
\end{rem}

Let $(G,\mathscr M,\alpha,\mu)$ be a $W^*$-dynamical system with $\alpha=\Ad(U)$ for a unitary representation $U\colon G\to \Iso(\mathscr H)$ as above and write $\nu=\mu\circ \lambda$. We denote by $S(r^*\mathscr M)$ the space of Borel sections $f\colon G\to r^*\mathscr M$ such that 
	\[\|f\|_I\coloneqq \max\left\lbrace \sup_{x\in G^{(0)}}\sum_{g\in G^x}\|f(g)\|,\sup_{x\in G^{(0)}}\sum_{g\in G_x}\|f(g)\|\right\rbrace<\infty.\]
	Then $S(r^*\mathscr M)$ is a $*$-algebra with respect to convolution and involution defined by
		\[f_1*f_2(g)\coloneqq \sum_{h\in G^{r(g)}}f_1(h)\alpha_h(f_2(h^{-1}g)),\quad f_1^*(g)\coloneqq \alpha_{g}(f_1(g^{-1})^*),\]
	for $f_1,f_2\in S(r^*\mathscr M)$ and $g\in G$. There is a representation
		\[\varphi_\alpha\colon S(r^*\mathscr M)\to \mathcal L(L^2(G,r^*\mathscr H,\nu^{-1}))\]
	given by 
		\[\varphi_\alpha(f)\xi(g)=\sum_{h\in G^{r(g)}}f(h)U_h(\xi(h^{-1}g)),\quad f\in S(r^*\mathscr M), \xi\in L^2(G,r^*\mathscr H,\nu^{-1}) ,g\in G.\]
		
\begin{defn}[cp. \cite{yamanouchiduality,OtyRamsay2006GvNA}]
	The crossed product $L^\infty(\mathscr M)\overline{\rtimes} G$ is the weak closure of $\varphi_\alpha(S(r^*\mathscr M))\subseteq \mathcal L(L^2(G,r^*\mathscr H,\nu^{-1}))$. 
\end{defn}

There is a canonical inclusion $L^\infty(\mathscr M)\hookrightarrow L^\infty(\mathscr M)\overline{\rtimes}G$ given by $m\mapsto \varphi_\alpha(\tilde m)$ where $\tilde m\in S(r^*\mathscr M)$ is given by 
	\[\tilde m(g)\coloneqq \begin{cases} m(g),&g\in G^{(0)}\\ 0, &g\notin G^{(0)}\end{cases}.\]

\begin{lem}
	Let $(G,\mathscr M,\alpha,\mu)$ be a $W^*$-dynamical system. The restriction map 
		\[\varphi_\alpha(S(r^*\mathscr M))\to L^\infty(\mathscr M),\quad \varphi_\alpha(f)\mapsto f|_{G^{(0)}}\]
	extends to a conditional expectation $E\colon L^\infty(\mathscr M)\overline{\rtimes}G\to L^\infty(\mathscr M)$. 
\end{lem}

\begin{proof}
	Although the lemma follows from a more general construction for $W^*$-Fell bundles over inverse semigroups in \cite{bussexelmeyer2017reduced}, we give a self-contained proof for the special case treated here. 
	Denote by $V\colon L^2(G^{(0)},\mathscr H,\mu)\to L^2(G,r^*\mathscr H,\nu^{-1})$ the isometry given by 
		\[(V\xi)(g)=\begin{cases}	\xi(g),&g\in G^{(0)}\\ 0,&g\notin G^{(0)}	\end{cases} \]
	for $\xi\in L^2(G^{(0)},\mathscr H,\mu)$ and $g\in G$. Then the map 
		\[\Ad(V^*)\colon \mathcal L(L^2(G,r^*\mathscr H,\nu^{-1}))\to\mathcal L(L^2(G^{(0)},\mathscr H,\mu)),\quad T\mapsto V^*TV\]
	is normal and completely positive. We denote its restriction to $\varphi_\alpha(S(r^*\mathscr M))$ by $E_0$. One easily checks that $E_0$ is given by $E_0(\varphi_\alpha(f))=f|_{G^{(0)}}$ for $f\in S(r^*\mathscr M))$ and that $E_0$ is $L^\infty(\mathscr M)$-bilinear. By normality, $E_0$ extends to a conditional expectation $E\colon L^\infty(\mathscr M)\overline{\rtimes }G\to L^\infty(\mathscr M)$. 
\end{proof}

\section{An approximation property for Fell bundles}
Throughout this section, let $G$ be a Hausdorff, second countable, \'etale groupoid.
For a Fell bundle $\mathscr B$ over $G$, we denote by $L^2(G,r^*\mathscr B^{(0)})$ the Hilbert-$C_0(\mathscr B^{(0)})$-module obtained by completing $C_c(G,r^*\mathscr B^{(0)})$ with respect to the $C_0(\mathscr B^{(0)})$- valued inner product 
	\[\langle \xi,\eta\rangle(x)\coloneqq \sum_{g\in G^x}\xi(g)^*\eta(g),\quad \xi,\eta\in L^2(G,r^*\mathscr B^{(0)}),~x\in \GG,\]
and equip it with the $C_0(\mathscr B^{(0)})$-module structure given by 
	\[\xi\cdot f(g)\coloneqq \xi(g)f(r(g)),\quad \xi\in L^2(G,r^*\mathscr B^{(0)}),f\in C_0(\mathscr B^{(0)}),g\in G.\]
 The following definition is a generalization of the approximation property of Exel \cite{Exel-Crelle} (see also \cite{exel2002approximation}).
\begin{defn}\label{defn:AP}
	A Fell bundle $\mathscr B$ over $G$ \emph{has the approximation property} if there exists a sequence of functions $(a_i)_{i\in \N}$ in $C_c(G,r^*\mathscr B^{(0)})\subseteq L^2(G,r^*\mathscr B^{(0)})$ such that 
	\begin{enumerate}
		\item $\sup_{i\in \N}\|\langle a_i,a_i\rangle \|<\infty$. \label{item:uniformly bounded}
		\item For every $f\in C_c(\mathscr B)$, the sequence $(f_i)_{i\in \N}$ given by
			\[f_i\colon G\to \mathscr B,\quad g\mapsto \sum_{h\in G^{r(g)}}a_i(h)^*f(g)a_i(g^{-1}h)\]
			converges to $f$ in the inductive limit topology. \label{item:almostinvariant}
	\end{enumerate}
\end{defn}

\begin{rem}
\begin{enumerate}
	\item The trivial Fell bundle $\mathscr B\coloneqq G\times \C$ has the approximation property if and only if $G$ is topologically amenable in the sense of \cite{ADRenault2000amenableGroupoids}. 
	\item If $G$ is topologically amenable, every Fell bundle $\mathscr B$ over $G$ has the approximation property. Indeed, if $(a_i)_{i\in \N}$ is a sequence in $C_c(G)$ witnessing the approximation property for $G\times \C$ and $(u_i)_{i\in \N}$ is an approximate unit for $C_0(\mathscr B^{(0)})$, then the sequence $(a_i\cdot (u_i\circ r))_{i\in \N}$ in $C_c(G,r^*\mathscr B^{(0)})$ witnesses the approximation property for $\mathscr B$.
\end{enumerate}
\end{rem}

\begin{thm}\label{thm:APimpliesWC}
	Let $\mathscr B$ be a Fell bundle over $G$ with the approximation property. Then we have $C^*(\mathscr B)\cong C^*_r(\mathscr B)$ via the regular representation.
\end{thm}

For the proof of Theorem \ref{thm:APimpliesWC}, we need 

\begin{lem}[{\cite[Lem.~2]{williamsamenability}}]\label{lem:williams}
	Let $(\mu,\mathscr H,\pi)$ be a strict representation of a Fell bundle $\mathscr B$ over $G$ and write $\nu\coloneqq \mu\circ \lambda$. Then the representation 
		\[M_{\pi}:C_c(\mathscr B)\to \mathcal L(L^2(G,r^*\mathscr H,\nu^{-1}))\]
	given by 
		\[M_\pi(f)\xi(g)\coloneqq \sum_{h\in G^{r(g)}}\pi(f(h))\xi(h^{-1}g),\quad f\in C_c(\mathscr B), \xi\in L^2(G,r^*\mathscr H,\nu^{-1}), g\in G\]
	is equivalent to the representation
	\[\Ind \pi\colon C_c(\mathscr B)\to \mathcal L(L^2(\mathscr B)\otimes_\pi L^2(\mathscr H))\]
	given by 
		\[\quad \Ind \pi(f)(\xi\otimes \eta)\coloneqq \Lambda(f)\xi\otimes \eta,\quad f\in C_c(\mathscr B),\xi \in L^2(\mathscr B), \eta\in L^2(\mathscr H).\]
	In particular, $M_\pi$ extends to $C^*_r(\mathscr B)$. 
\end{lem}

\begin{proof}[Proof of Theorem \ref{thm:APimpliesWC}]
	By Theorem \ref{thm:Fell-disintegration} it suffices to prove that the integrated form $L_\pi$ of every strict representation $(\mu, \mathscr H,\pi)$ factors through $C^*_r(\mathscr B)$. Let $(a_i)_{i\in \N}$ be a sequence in $C_c(G,r^*\mathscr B^{(0)})$ as in Definition \ref{defn:AP}. Write $\nu\coloneqq \mu\circ\lambda$ and $\Delta\coloneqq \frac{d\nu}{d\nu^{-1}}$. For each $i\in \N$, define an operator $T_i\colon L^2(G^{(0)},\mathscr H,\mu)\to L^2(G,r^*\mathscr H,\nu^{-1})$ by 
	\[T_i\xi(g)\coloneqq \Delta^{\frac 1 2}(g)\pi(a_i(g))\xi(r(g)),\quad \xi\in L^2(\GG,\mathscr H,\mu),~g\in G.\]
	The adjoint of $T_i$ is given by
		\[T_i^*\eta(x)=\sum_{g\in G^x}\pi(a_i(g)^*)\eta(g)\Delta^{-\frac 1 2}(g),\quad \eta\in L^2(G,r^*\mathscr H,\nu^{-1}),~ x\in G^{(0)}.\]
	Note that
		\[V_i\colon C_c(\mathscr B)\to \mathcal L(L^2(G^{(0)},\mathscr H,\mu)),\quad V_i(f)\coloneqq T_i^*M_\pi(f)T_i\]
	extends to $C^*_r(\mathscr B)$ by Lemma \ref{lem:williams}. 
	We claim that for each $f\in C_c(\mathscr B)$ we have $V_i(f)\xrightarrow{i\to\infty} L_\pi(f)$ in norm. Fix $\xi\in L^2(G^{(0)},\mathscr H,\mu)$. For almost every $x\in G^{(0)}$, we have 
	\begin{align*}
		V_i(f)\xi(x)= &T_i^*M_\pi(f)T_i\xi(x)\\
		=&\sum_{g\in G^x}\pi(a_i(g))^*\Big((M_\pi(f)T_i\xi)(g)\Big)\Delta^{-\frac 1 2}(g)\\
		=&\sum_{g\in G^x}\pi(a_i(g))^*\left(\sum_{h\in G^x}\pi(f(h))(T_i\xi)(h^{-1}g)\right)\Delta^{-\frac 1 2}(g)\\
		=&\sum_{g\in G^x}\pi(a_i(g))^*\left(\sum_{h\in G^x}\pi(f(h))\pi(a_i(h^{-1}g))\xi(s(h))\Delta^{\frac 1 2}(h^{-1}g)\right)\Delta^{-\frac 1 2}(g)\\
		=&\sum_{h\in G^x}\pi\underbrace{\left(\sum_{g\in G^x}a_i(g)^*f(h)a_i(h^{-1}g)\right)}_{=f_i(h)}\xi(s(h))\Delta^{-\frac 1 2}(h)\\
		=&L_\pi(f_i)\xi(x).
	\end{align*}
	 Here, we have used Lemma \ref{lem:hom-almost-everywhere} at the second to last equality. By the second condition of Definition \ref{defn:AP}, we have $f_i\to f$ in the inductive limit topology, thus we have $V_i(f)\to L_\pi(f)$ in norm as claimed. 
	 It follows from the first condition of Definition \ref{defn:AP} that the $(V_i)_{i\in \N}$ are uniformly bounded. Thus by a $\frac \varepsilon 3$-argument we get $V_i(f)\to L_\pi(f)$ in norm for every $f\in C^*(\mathscr B)$. This proves that $M_\pi$ factors through $C^*_r(\mathscr B)$. 
\end{proof}

\begin{thm}\label{thm:APimpliesnuclearity}
	Let $\mathscr B$ be a Fell bundle over $G$ with the approximation property. Assume that $C_0(\mathscr B^{(0)})$ is nuclear. Then $C^*_r(\mathscr B)$ is nuclear as well. 
\end{thm}

Theorem \ref{thm:APimpliesnuclearity} has been proved for actions of locally compact, Hausdorff, second countable, \emph{amenable} groupoids on separable $C^*$-algebras in \cite{LalondeNuclearity} 
and for \emph{continuous} Fell bundles over locally compact, Hausdorff, \'etale, \emph{amenable} groupoids in \cite{takeishi}. 
It was later generalized to (upper semicontinuous) Fell bundles over locally compact, Hausdorff, second countable, amenable groupoids in \cite{Lal19} by combining the result from \cite{LalondeNuclearity} with the stabilization theorem of \cite{WilliamsStabilization}.
We will follow the proof of \cite{LalondeNuclearity}, and use an upper semicontinuous version of the tensor product construction from \cite{takeishi}. 
The main steps of the proof are to
\begin{enumerate}
	\item construct tensor products of Fell bundles with $C^*$-algebras;
	\item establish compatibility of minimal tensor products with reduced cross-sectional algebras;
	\item establish compatibility of maximal tensor products with full cross-sectional algebras;
	\item show that the approximation property passes to tensor products;
	\item use Theorem \ref{thm:APimpliesWC} to finish the proof. 
\end{enumerate}

We begin with the construction of tensor products. Let $\mathscr B$ be a Fell bundle over $G$ and let $A$ be a separable $C^*$-algebra. For $g\in G$, denote by $\mathscr B_g\otimes A$ the completion of the algebraic tensor product $\mathscr B_g\odot A$ by the $\mathscr B_{s(g)}\otimes A$-valued right inner product 
	\[\langle b_1\otimes a_1,b_2\otimes a_2\rangle_{\mathscr B_{s(g)}\otimes A}\coloneqq b_1^*b_2\otimes a_1^*a_2,\quad b_1,b_2\in \mathscr B_g,~a_1,a_1\in A,\]
or equivalently by the $\mathscr B_{r(g)}\otimes A$-valued left inner product
\[{}_{\mathscr B_{r(g)}\otimes A}\langle b_1\otimes a_1,b_2\otimes a_2\rangle\coloneqq b_1b_2^*\otimes a_1a_2^*,\quad b_1,b_2\in \mathscr B_g,~a_1,a_1\in A.\]
Note that $\mathscr B_g\otimes A$ is the minimal exterior tensor product of the $\mathscr B_{r(g)}$-$\mathscr B_{s(g)}$-imprimitivity bimodule $\mathscr B_g$ with $A$ (see \cite[Chap.~4]{Lance}). We define the maximal tensor product $\mathscr B_g\otimes_{\max}A$ analogously, using inner products in $\mathscr B_{s(g)}\otimes_{\max}A$ and $\mathscr B_{r(g)}\otimes_{\max}A$. 

\begin{lem}\label{lem:tensor product is fell bundle}
	For every $(g,h)\in G^{(2)}$, the multiplication 
		\begin{equation}\label{eq:multiplication on tensor product1}
			(\mathscr B_g\odot A)\odot (\mathscr B_h\odot A)\to \mathscr B_{gh}\odot A,\quad (b\otimes a)\otimes (b'\otimes a')\mapsto bb'\otimes aa'
		\end{equation}
	extends continuously to a map
		\begin{equation}\label{eq:multiplication on tensor product}
			(\mathscr B_g\otimes A)\otimes_{\mathscr B_{s(g)}\otimes A} (\mathscr B_h\otimes A)\to \mathscr B_{gh}\otimes A.
		\end{equation}
	The involution
		\[\mathscr B_g\odot A\to \mathscr B_{g^{-1}}\odot A,\quad b\otimes a\mapsto b^*\otimes a^*\]
	extends continuously to a map 
		\[\mathscr B_g\otimes A\to \mathscr B_{g^{-1}}\otimes A.\]
	The analogous statement holds for maximal tensor products. 
\end{lem}
\begin{proof}
	By plugging in elementary tensors, one easily checks that the maps \eqref{eq:multiplication on tensor product1} and \eqref{eq:multiplication on tensor product} preserve inner products and are thus isometric with respect to the Hilbert module norms. 
\end{proof}

Using Lemma \ref{lem:tensor product is fell bundle}, we equip the bundles $\mathscr B\otimes A\coloneqq \{\mathscr B_g\otimes A\}_{g\in G}$ and $\mathscr B\otimes_{\max}A\coloneqq \{\mathscr B_g\otimes_{\max}A\}_{g\in G}$ with the structure of an algebraic Fell bundle. Note that every $f\in C_c(\mathscr B)$ and $a\in A$ define a compactly supported section $f\otimes a$ of $\mathscr B\otimes A$ and of $\mathscr B\otimes_{\max}A$. We use these sections to define a topology on $\mathscr B\otimes A$ and $\mathscr B\otimes_{\max}A$. 

\begin{lem}
There is a unique topology on $\mathscr B\otimes_{\max}A$ turning it into a Fell bundle such that $C_c(\mathscr B)\odot A\subseteq C_c(\mathscr B\otimes_{\max}A)$ is a subalgebra of continuous sections which is dense in the inductive limit topology. 
If moreover $C_0(\mathscr B^{(0)})$ is exact, the same statement holds for the minimal tensor product $\mathscr B\otimes A$.
\end{lem}
\begin{proof}
	We check the assumptions of Proposition \ref{prop:algebraic Fell bundles}. Condition \eqref{item:fiberwise dense} follows from the fact that every point in $\mathscr B$ is the image of a continuous section $f\in C_c(\mathscr B)$ \cite[Prop. 3.4]{HofmannBundles}. For the maximal tensor product, condition \eqref{item:upper semicontinuous} follows from \cite[Lem.~2.3]{kirchberg1995operations}. For the minimal tensor product, condition \eqref{item:upper semicontinuous} follows from \cite[Lem.~2.3]{kirchberg1995operations} and \cite[Prop.~3.1]{LalondeNuclearity}.
\end{proof}
	
\begin{lem}\label{lem:reducedandminimalproduct}
	Let $\mathscr B$ be a Fell bundle over $G$ and let $A$ be a separable $C^*$-algebra. Assume that $C_0(\mathscr B^{(0)})$ is exact. Then the identity on $C_c(\mathscr B)\odot A$ extends to an isomorphism $C^*_r(\mathscr B)\otimes A\cong C^*_r(\mathscr B\otimes A)$. 
\end{lem}
\begin{proof}
	The exactness assumption on $C_0(\mathscr B^{(0)})$ is only used for the construction of $\mathscr B\otimes A$. 
	Fix faithful non-degenerate representations $\pi\colon C_0(\mathscr B^{(0)})\to \mathcal L(H_1)$ and $\rho\colon A\to \mathcal L(H_2)$. We obtain faithful non-degenerate representations
	\[(\Ind \pi)\otimes \rho\colon C^*_r(\mathscr B)\otimes A\hookrightarrow\mathcal L((L^2(\mathscr B)\otimes_\pi H_1)\otimes H_2\]
		and 
	\[\Ind(\pi\otimes \rho)\colon C^*_r(\mathscr B\otimes A)\hookrightarrow \mathcal L(L^2(\mathscr B\otimes A)\otimes_{\pi\otimes\rho}(H_1\otimes H_2)),\]
	where $\Ind \pi$ and $\Ind(\pi\otimes \rho)$ are defined as in Lemma \ref{lem:williams}. We claim that there is a unitary operator 
	\[U\colon L^2(\mathscr B\otimes A)\otimes_{\pi\otimes\rho}(H_1\otimes H_2)\to (L^2(\mathscr B)\otimes_\pi H_1)\otimes H_2 \]
	satisfying 
		\begin{equation}\label{eq:U on dense subspace}
			U((f\otimes a)\otimes (\xi\otimes \eta))\coloneqq (f\otimes \xi)\otimes \rho(a)\eta 
		\end{equation}
	for all $f\in C_c(\mathscr B),~a\in A,~\xi\in H_1,~\eta\in H_2$. Indeed, since $C_c(\mathscr B)\odot A\subseteq C_c(\mathscr B\otimes A)$ is dense in the inductive limit topology, $C_c(\mathscr B)\otimes A\subseteq L^2(\mathscr B\otimes A)$ is dense in norm. Therefore \eqref{eq:U on dense subspace} defines $U$ on a dense subspace. A calculation on elementary tensors shows that $U$ preserves inner products and is therefore isometric. Since $\rho$ is non-degenerate, we conclude that $U$ has dense image and is thus a unitary. 
	Another calculation on elementary tensors shows that we have 
		\[((\Ind \pi)\otimes \rho) (f\otimes a)U=U \Ind(\pi\otimes \rho)(f\otimes a)\]
	for all $f\in C_c(\mathscr B)$ and $a\in A$. Therefore $\Ad(U)$ restricts to the desired isomorphism $C^*_r(\mathscr B)\otimes A\cong C^*_r(\mathscr B\otimes A)$.
\end{proof}

\begin{lem}\label{lem:maximalandmaximalproduct}
	Let $\mathscr B$ be a Fell bundle over $G$ and let $A$ be a unital separable $C^*$-algebra. Then the identity on $C_c(\mathscr B)\odot A$ extends to an isomorphism \[C^*(\mathscr B)\otimes_{\max}A\cong C^*(\mathscr B\otimes_{\max}A).\]
\end{lem}
\begin{proof}
	We prove the Lemma by constructing $*$-homomorphisms $\Phi\colon C^*(\mathscr B)\otimes_{\max}A\to C^*(\mathscr B\otimes_{\max}A)$ and $\Psi\colon C^*(\mathscr B\otimes_{\max}A)\to C^*(\mathscr B)\otimes_{\max}A$ that extend the identity on $C_c(\mathscr B)\odot A$. 
	
	\subsubsection*{Construction of $\Phi$:}
Denote by $\iota\colon C_c(\mathscr B)\odot A\hookrightarrow C^*(\mathscr B\otimes_{\max}A)$ the canonical inclusion. Since $G$ is \'etale, we have a non-degenerate inclusion 
	\[C_0(\mathscr B^{(0)})\otimes_{\max} A=C_0\left((\mathscr B\otimes_{\max}A)^{(0)}\right)\subseteq C^*(\mathscr B\otimes_{\max}A).\]
There is an approximate unit for $C^*(\mathscr B\otimes_{\max}A)$ in $C_c(\mathscr B^{(0)})\otimes 1$, in particular there is an approximate unit in $C_c(\mathscr B)\otimes 1$. 
Now the standard arguments as in \cite[Lem. 6.3.4, Thm. 6.3.5]{Murphy} imply that we can find commuting non-degenerate $*$-homomorphisms $\iota_\mathscr B\colon C_c(\mathscr B)\to C^*(\mathscr B\otimes_{\max}A)$ and $\iota_A\colon A\to M(C^*(\mathscr B\otimes_{\max}A))$ such that 
	\[\iota(f\otimes a)=\iota_\mathscr B(f)\iota_A(a)=\iota_A(a)\iota_\mathscr B(f),\quad \forall f\in C_c(\mathscr B), ~a\in A.\]
Denote by $\overline {\pi_\mathscr B}$ the extension of $\pi_\mathscr B$ to the enveloping algebra $C^*(\mathscr B)$ of $C_c(\mathscr B)$. By the universal property of the maximal tensor product, there is a unique $*$-homomorphism 
	\[\Phi\colon C^*(\mathscr B)\otimes_{\max}A\to C^*(\mathscr B\otimes_{\max}A)\]
such that 
	\[\Phi(f\otimes a)=\overline{\pi_\mathscr B}(f)\pi_A(a)=\pi_A(a)\overline{\pi_\mathscr B}(f),\quad \forall f\in C^*(\mathscr B), a\in A.\]
	In particular, $\Phi$ extends the identity on $C_c(\mathscr B)\odot A$.
	
\subsubsection*{Construction of $\Psi$:}
Fix a faithful, non-degenerate representation $\Pi\colon C^*(\mathscr B)\otimes_{\max} A\hookrightarrow \mathcal L(H)$ on a separable Hilbert space $H$. Then there are non-degenerate representations $\rho\colon A\to \mathcal L(H)$ and $L\colon C^*(\mathscr B)\to \mathcal L(H)$ such that 
	\begin{equation}\label{eq: L and rho}
		\Pi(f\otimes a)= L(f)\rho(a)=\rho(a)L(f),\quad \forall f\in C^*(\mathscr B), a\in A.
	\end{equation}
By Theorem \ref{thm:Fell-disintegration}, we can identify $L$ with the integrated form $L_\pi$ of a strict representation $(\mu,\mathscr H,\pi)$ of $\mathscr B$ and write $H=L^2(G^{(0)},\mathscr H,\mu)$. Since $\rho(A)$ commutes with $L(C_0(\mathscr B^{(0)}))$, it also commutes with the image of $C_0(G^{(0)})$ and therefore $\rho(A)$ commutes with the diagonal operators on $L^2(G^{(0)},\mathscr H,\mu)$. By Proposition \ref{direct integral representation}, there is a Borel family of representations $(\rho_x\colon A\to \mathcal L(\mathscr H_x))_{x\in G^{(0)}}$ such that $\rho=\int_{G^{(0)}}\rho_xd\mu(x)$. 

\begin{claim}\label{clm: commute on g}
	Denote by $\mathcal U\subseteq G^{(0)}$ the set of all $x\in G^{(0)}$ such that for all $g\in G_x$, we have
		\begin{equation}\label{eq: commute on g}
			\pi_g(b)\rho_{s(g)}(a)=\rho_{r(g)}(a)\pi_g(b),\quad \forall b\in \mathscr B_g, a\in A.
		\end{equation}
	Then $\mathcal U$ is a $G$-invariant Borel set with $\mu(G^{(0)}\setminus \mathcal U)=0$.
\end{claim}

To see that $\mathcal U$ is $G$-invariant, let $x\in \mathcal U$ and $g\in G_x$. Then for any $h\in G_{r(g)}, b\in \mathscr B_h,b'\in \mathscr B_g$ and $a\in A$, we have 
			\begin{align*}
				\pi_h(b)\rho_{s(h)}(a)\pi_g(b')=&\pi_h(b)\pi_g(b')\rho_{s(g)}(a)=\pi_{hg}(bb')\rho_{s(g)}(a)\\
				=&\rho_{r(h)}(a)\pi_{hg}(bb')=\rho_{r(h)}(a)\pi_h(b)\pi_g(b'),
			\end{align*}
where we have use $x\in \mathcal U$ at the first and third equation.
Since $\mathscr B$ is saturated (Condition \eqref{item:saturated} in Definition \ref{defn:Fell bundle}) and since $b'\in \mathscr B_g$ was arbitrary, this implies $\pi_h(b)\rho_{s(h)}(a)=\rho_{r(h)}(a)\pi_h(b)$. Therefore $\mathcal U$ is $G$-invariant. 

We now show that $\mathcal U$ is Borel with $\mu(G^{(0)}\setminus \mathcal U)=0$. Fix a dense sequence $(a_n)_{n\in \N}$ in $A$ and a sequence $(U_k)_{k\in \N}$ of open bisections in $G$ with $G=\bigcup_{k\in \N}U_k$. For each $k\in \N$, fix a sequence $(f_m^{(k)})_{m\in \N}$ of sections in $C_c(U_k,\mathscr B)$ such that for each $g\in U_k$, the set $\{f_m^{(k)}(g)\mid m\in \N\}\subseteq \mathscr B_g$ is dense. 
By the choices of $(a_n)_{n\in \N}, (U_k)_{k\in \N}$ and $(f_m^{(k)})_{m\in \N}$, we have 

			\begin{equation}\label{eq: formula for mathcal U}
			G^{(0)}\setminus\mathcal U=\bigcup_{k\in \N} s\left(\bigcup_{n,m\in \N}\left\lbrace g\in U_k \mid \pi_g(f^{(k)}_m(g))\rho_{s(g)}(a_n)\not=\rho_{r(g)}(a_n)\pi_g(f^{(k)}_m(g))\right\rbrace\right).
			\end{equation}
For fixed $n,k,m\in \N$, \eqref{eq: L and rho} implies $L(f^{(k)}_m)\rho(a_n)=\rho(a_n)L(f^{(k)}_m)\in \mathcal L(L^2(G^{(0)},\mathscr H,\mu))$. Both $L(f^{(k)}_m)\rho(a_n)$ and $\rho(a_n)L(f^{(k)}_m)$ restrict to decomposable operators 
	\[L^2(s(U_k),\mathscr H|_{s(U_k)},\mu)\to L^2(r(U_k),\mathscr H|_{r(U_k)},\mu),\]
		where we view $\mathscr H|_{r(U_k)}$ as a Borel Hilbert bundle over $s(U_k)$ via the canonical homeomorphism $s(U_k)\cong U_k\cong r(U_k)$. Using Theorem \ref{Thm unique decomposition}, this yields
			\begin{align*}
				&\pi_g(f^{(k)}_m(g))\rho_{s(g)}(a_n)=\Delta^{\frac 1 2}(g)(L(f^{(k)}_m)\rho(a_n))_{s(g)}\\
				=&\Delta^{\frac 1 2}(g)(\rho(a_n)L(f^{(k)}_m))_{s(g)}=\rho_{r(g)}(a_n)\pi_g(f^{(k)}_m(g))
			\end{align*}
		for $r^*\mu$-almost all $g\in U_k$. 
		This together with \eqref{eq: formula for mathcal U} proves that $\mathcal U$ is Borel with $\mu(G^{(0)}\setminus \mathcal U)=0$. This establishes Claim \ref{clm: commute on g}. \\

\begin{claim}\label{claim: extend continuously}
	For each $g\in G|_\mathcal U$, the map
			\begin{equation}\label{eq:times star functor}
				\pi_g\times \rho_{s(g)}\colon\mathscr B_g\odot A\to \mathcal{L}(\mathscr H_{s(g)},\mathscr H_{r(g)}),\quad \pi_g\times \rho_{s(g)}(b\otimes a)\coloneqq\pi_g(b)\rho_{s(g)}(a)
			\end{equation}
		extends continuously to $\mathscr B_g\otimes_{\max}A$.
\end{claim}
		Indeed, for $\sum_i b_i\odot a_i\in \mathscr B_g \odot A$ we have
			\begin{align*}
				&\left\|\pi_g\times \rho_{s(g)}\left(\sum_i b_i\otimes a_i\right)\right\|^2
				= \left\|\sum_{i,j}\rho_{s(g)}(a_i^*)\pi_{s(g)}(b_i^*b_j)\rho_{s(g)}(a_j)\right\| \\
				=&\left\| \sum_{i,j} \pi_{s(g)}(b_i^*b_j)\rho_{s(g)}(a_i^*a_j)\right\|
				=\left\| \pi_{s(g)}\times \rho_{s(g)}\left(\left(\sum_i b_i\otimes a_i\right)^*\left(\sum_j b_j\otimes a_j\right)\right)\right\| \\
				\leq &\left\|\left(\sum_i a_i\otimes b_i\right)^*\left(\sum_j a_j\otimes b_j\right)\right\|_{\mathscr B_{s(g)}\otimes_{\max}A}=\left \|\sum_i b_i\otimes a_i\right\|^2_{\mathscr B_g\otimes_{\max}A},
			\end{align*}
		where we have used $g\in G|_{\mathcal U}$ at the second equality. This proves Claim \ref{claim: extend continuously}.

		Note that since $\mathcal U\subseteq G^{(0)}$ is $G$-invariant, $G|_{\mathcal U}\coloneqq r^{-1}(\mathcal U)$ is a Borel subgroupoid whose complement $G\setminus G|_{\mathcal U}=G|_{\mathcal G^{(0)}\setminus \mathcal U}$ is a Borel subgroupoid of measure zero. 
		We define $\pi_g\times \rho_{s(g)}$ as in \eqref{eq:times star functor} for $g\in G|_{\mathcal U}$ and $\pi_g\times \rho_{s(g)}\coloneqq 0$ for $g\in G|_{G^{(0)}\setminus\mathcal U}$. It now easily follows from Claim \ref{clm: commute on g} and the above that $\pi\times \rho\coloneqq\{\pi_g\times \rho_{s(g)}\}_{g\in G}$ defines a $*$-functor 
			\begin{equation}\label{eq:pitimesrho}
				\pi\times \rho\colon\mathscr B\otimes_{\max} A\to \End(\mathscr H).
			\end{equation}
		\begin{claim}\label{claim:Borel}
			$\pi \times \rho$ is Borel.
		\end{claim}
		Let $\xi,\eta\in B(G^{(0)},\mathscr H)$ be Borel sections and let $f\in C_c(\mathscr B\otimes_{\max}A)$. We have to prove that the map
			\[G\ni g\mapsto c(g)\coloneqq\langle \pi\times \rho(f(g))\xi(s(g)),\eta(r(g))\rangle\in \C \]
		is Borel. Since by construction of $\mathscr B\otimes_{\max}A$, we can write $f$ as a norm limit of a sequence in $C_c(\mathscr B)\odot A$, we can assume $f\in C_c(\mathscr B)\odot A$. By linearity, we can assume that $f=h\otimes a$ for $h\in C_c(\mathscr B)$ and $a\in A$. But then 
			\[c(g)=\begin{cases}
			\langle \pi_g(h(g))\xi(s(g)),\rho_{r(g)}(a^*)\eta(r(g))\rangle, & g\in G|_{\mathcal U}\\ 0,&g\notin G|_{\mathcal U}
			\end{cases}\]
		is Borel in $g\in G$ since $\pi$ is a Borel $*$-functor and since $\{\rho_x\}_{x\in G^{(0)}}$ is a Borel family of representations. This proves Claim \ref{claim:Borel}.
		
		Now denote by 
			\[\Psi\colon C^*(\mathscr B\otimes_{\max}A)\to \mathcal L(L^2(G^{(0)},\mathscr H,\mu))\]
		the integrated form of the Borel-$*$-functor $\pi\times \rho$ as defined in \eqref{eq:pitimesrho}. By construction we have $\Psi(f\otimes a)=\Pi(f\otimes a)$ for every $f\in C_c(\mathscr B)$ and $a\in A$. In particular, $\Psi$ maps onto $\Pi(C^*(\mathscr B)\otimes_{\max}A)\cong C^*(\mathscr B)\otimes_{\max}A$ and has the desired properties. 
\end{proof}

\begin{lem}\label{lem:easynuclearity}
	Let $B$ be a $C^*$-algebra. Then $B$ is nuclear if and only if for every unital, separable $C^*$-algebra $A$, we have $A\otimes_{\max}B=A\otimes B$. 
\end{lem}
\begin{proof}
	By \cite[Prop. 5.2]{LalondeNuclearity}, it suffices to check $A\otimes_{\max}B=A\otimes B$ for all separable $A$. Passing to the unitization of $A$ and using that tensor products preserve ideal inclusions reduces the situation to the case that $A$ is unital.
\end{proof}

\begin{lem}\label{lem:APandtensorproducts}
	Let $\mathscr B$ be a Fell bundle over $G$ with the approximation property and let $A$ be a unital, separable $C^*$-algebra. Then $\mathscr B\otimes_{\max}A$ has the approximation property. 
	If additionally $C_0(\mathscr B^{(0)})$ is exact, then $\mathscr B\otimes A$ has the approximation property as well. 
\end{lem}
\begin{proof}
	We only prove the statement about the maximal tensor product. The proof for the minimal tensor product is identical, the exactness assumption is only used in the construction of $\mathscr B\otimes A$. Let $(a_i)_{i\in \N}$ be a sequence in $C_c(G,r^*\mathscr B^{(0)})$ as in Definition \ref{defn:AP}. We claim that the sequence $(a_i\otimes 1)_{i\in \N}$ in $C_c(G,r^*\mathscr B^{(0)}\otimes_{\max} A)$ satisfies both conditions of Definition \ref{defn:AP} for the Fell bundle $\mathscr B\otimes_{\max} A$. 
	
	Condition \eqref{item:uniformly bounded} is clearly satisfied. To check condition \eqref{item:almostinvariant} of Definition \ref{defn:AP}, we first fix an element $f=\sum_{k=1}^n x_k\otimes b_k\in C_c(\mathscr B)\odot A\subseteq C_c(\mathscr B\otimes_{\max} A)$ and define 
		\begin{equation}\label{eq:defn fi}
			f_i(g)\coloneqq \sum_{h\in G^{r(g)}}(a_i\otimes 1)(h)^*f(g)(a_i\otimes 1)(g^{-1}h)
		\end{equation}
	for $g\in G$ and $i\in \N$. Then we have 
		\[\|f(g)-f_i(g)\|\leq \sum_{k=1}^n\underbrace{\left \|x_k(g)-\sum_{h\in G^{r(g)}}a_i(h)^*x_k(g)a_i(g^{-1}h)\right\|}_{\xrightarrow{i\to\infty} 0}\|b_k\|\xrightarrow{i\to\infty}0\]
	where the convergence is uniform in $g\in G$. This shows that we have $f_i\to f$ in the inductive limit topology and verifies condition \eqref{item:almostinvariant} of Definition \ref{defn:AP} for $f\in C_c(\mathscr B)\odot A$. 
	For a general $\tilde f\in C_c(\mathscr B\otimes_{\max} A)$ and $\varepsilon>0$, we pick an element $f\in C_c(\mathscr B)\odot A$ satisfying $\|f(g)-\tilde f(g)\|<\frac \varepsilon 3$ for all $g\in G$. We denote by $\tilde f_i$ the same function as in \eqref{eq:defn fi} with $f(g)$ replaced by $\tilde f(g)$ and write 
		\[C\coloneqq \max\left\lbrace 1,\sup_{i\in \N}\|a_i\|_{L^2(G,r^*\mathscr B^{(0)})}^2\right\rbrace.\]
	 Now let $i\in \N$ be large enough such that $\|f(g)-f_i(g)\|<\frac \varepsilon {3 C}$ for all $g\in G$. Then using the Cauchy-Schwarz-inequality for Hilbert modules, we get 
	\begin{align*}
		\|\tilde f(g)-\tilde f_i(g)\|\leq& \|\tilde f(g)-f(g)\|+\|f(g)-f_i(g)\|+\|f_i(g)-\tilde f_i(g)\| \\
		<&\frac 2 3 \varepsilon+\left\|\sum_{h\in G^{r(g)}}(a_i\otimes 1)^*(h)(f(g)-\tilde f(g))(a_i\otimes 1)(g^{-1}h)\right\| \\
		=&\frac 2 3 \varepsilon + \|\langle (a_i\otimes 1)|_{G^{r(g)}},(f(g)-\tilde f(g))(a_i\otimes 1)|_{G^{s(g)}}\rangle \| \\
		\leq &\frac 2 3 \varepsilon + \|a_i|_{G^{r(g)}}\|_{L^2(G^{r(g)},\mathscr B_{r(g)})} \cdot \|(f(g)-\tilde f(g))\| \cdot \|a_i|_{G^{s(g)}}\|_{L^2(G^{s(g)},\mathscr B_{s(g)})} \\
		\leq &\frac 2 3 \varepsilon + C \frac \varepsilon {3C}\\
		=&\varepsilon. 
	\end{align*}
	This verifies condition \eqref{item:almostinvariant} of Definition \ref{defn:AP} for general $\tilde f\in C_c(\mathscr B\otimes_{\max}A)$. 
\end{proof}

\begin{proof}[Proof of Theorem \ref{thm:APimpliesnuclearity}]
	Let $\mathscr B$ be a Fell bundle with the approximation property such that $C_0(\mathscr B^{(0)})$ is nuclear. Let $A$ be a unital, separable $C^*$-algebra. 
	Since the norms on both $\mathscr B\otimes_{\max}A$ and $\mathscr B\otimes A$ only depend on the unit bundles $\mathscr B^{(0)}\otimes_{\max}A$ and $\mathscr B^{(0)}\otimes A$, and since $\mathscr B^{(0)}$ has nuclear fibers, we have 
		\begin{equation}\label{eq:B max A equal B min A}
			\mathscr B\otimes_{\max}A=\mathscr B\otimes A.
		\end{equation}
		Using all the previous lemmas, we get isomorphisms
			\begin{align*}
				C^*(\mathscr B)\otimes_{\max}A\cong C^*(\mathscr B\otimes_{\max}A)\cong C^*_r(\mathscr B\otimes_{\max}A)\\
				\cong C^*_r(\mathscr B\otimes A)\cong C^*_r(\mathscr B)\otimes A\cong C^*(\mathscr B)\otimes A,
			\end{align*}
	all of which extend the identity on $C_c(\mathscr B)\odot A$. 
	Here we have used Lemma \ref{lem:maximalandmaximalproduct}, Lemma \ref{lem:APandtensorproducts}, \eqref{eq:B max A equal B min A}, Lemma \ref{lem:reducedandminimalproduct}, and Theorem \ref{thm:APimpliesWC} in this order.
	Now Lemma \ref{lem:easynuclearity} implies that $C^*(\mathscr B)$ is nuclear. 
\end{proof}

\section{Amenability for groupoid actions}
Throughout this section, let $G$ be a Hausdorff, second countable, \'etale groupoid.

\subsection{Amenability for $W^*$-dynamical systems}\label{sec:amenable vN}
Let $(G,\mathscr M,\alpha,\mu)$ be a $W^*$-dy\-namical system. We denote by $L^2(G,r^*\mathscr M)$ the space of equivalence classes of measurable sections $\xi\colon G\to r^*\mathscr M$ (with functions identified if they agree $\mu$-almost everywhere) such that 

\begin{enumerate}
	\item For $\mu$-almost all $x\in G^{(0)}$, the series $\sum_{g\in G^x}\xi(g)^*\xi(g)$ converges ultraweakly in $\mathscr M_x$.
	\item The map 
	\[G^{(0)}\ni x\mapsto \sum_{g\in G^x}\xi(g)^*\xi(g)\in \mathscr M_x\]
is $\mu$-essentially bounded.
\end{enumerate} 
Note that $L^2(G,r^*\mathscr M)$ is a Hilbert module over $L^\infty(\mathscr M)$ with the $L^\infty(\mathscr M)$-module structure given by pointwise multiplication and with $L^\infty(\mathscr M)$-valued inner products given by 
	\[\langle \xi,\eta\rangle(x)\coloneqq \sum_{g\in G^x}\xi(g)^*\eta(g),\quad \xi,\eta\in L^2(G,r^*\mathscr M), x\in G^{(0)}.\]
We write $Z(\mathscr M)\coloneqq \bigsqcup_{x\in G^{(0)}}Z(\mathscr M_x)$ and denote by $(G,Z(\mathscr M),\alpha,\mu)$ the corresponding $W^*$-dynamical system.

\begin{defn}\label{defn:measurewise amenable}
	A $W^*$-dynamical system $(G,\mathscr M,\alpha,\mu)$ is called \emph{amenable} if there is a sequence $(a_i)_{i\in \N}$ in $L^2(G,r^*Z(\mathscr M))$ such that 
		\begin{enumerate}
			\item $\sup_{i\in \N}\|\langle a_i,a_i\rangle \|_{L^\infty(Z(\mathscr M))}<\infty$ \label{item:W*almost 1}
			\item The sequence $(\tilde a_i)_{i\in \N}$ in $L^\infty(G,r^*Z(\mathscr M),\mu\circ \lambda)$ given by 
				\[\tilde a_i(g)\coloneqq \sum_{h\in G^{r(g)}}a_i(h)^*\alpha_g(a_i(g^{-1}h)),\quad g\in G\]
				converges to $1$ in the weak-$*$-topology of $L^\infty(G,r^*Z(\mathscr M),\mu\circ\lambda)$. \label{item:W* almost inv}
		\end{enumerate}
	We call a $G$-von Neumann algebra $(M,\alpha,\mu)$ \emph{amenable} if the associated $W^*$-dynamical system $(G,\mathscr M,\alpha,\mu)$ is amenable. 
\end{defn}

\begin{eg}
	Let $\mu$ be a quasi-invariant measure on $G^{(0)}$. Then the trivial $G$-von Neumann algebra $(L^\infty(G^{(0)},\mu),\alpha,\mu)$ is amenable in the sense of Definition \ref{defn:measurewise amenable} if and only if the measured groupoid $(G,\mu)$ is amenable in the sense of \cite{ADRenault2000amenableGroupoids}. 
\end{eg}

\begin{eg}\label{eg:Linfty(G,M)}
	Let $(G,\mathscr M,\alpha,\mu)$ be a $W^*$-dynamical system. Assume that $\mathscr M\subseteq \End(\mathscr H)$ and that $\alpha=\Ad(U)$ for a unitary representation $U\colon G\to \Iso(\mathscr H)$ as in Remark \ref{rem:standard form}. We construct a new $W^*$-dynamical system $(G,\tilde {\mathscr M},\tilde \alpha,\mu)$ as follows. 
	Let $(\xi_n)_{n\in \N}$ be a fundamental sequence for $\mathscr H$ and let $(f_m)_{m\in \N}$ be a sequence in $C_c(G)$ that is dense in the inductive limit topology. Using Theorem \ref{thm:construct Hilbert bundle}, we turn
	 $\tilde {\mathscr H}\coloneqq \bigsqcup_{x\in G^{(0)}}\ell^2(G^x, \mathscr H_x)$ into a Hilbert bundle over $G^{(0)}$ by declaring $(f_m\otimes \xi_n)_{n,m\in \N}$ a fundamental sequence. We define a unitary representation 
	 \[\tilde U\colon G\to \Iso(\tilde{\mathscr H}),\quad \tilde U_g(\xi)(h)\coloneqq U_g(\xi(g^{-1}h))\]
	 for  $g\in G, \xi\in \ell^2(G^{s(g)},\mathscr H_{s(g)}),$ and $h\in G^{r(g)}$.
	 Write 
	 	\[\tilde {\mathscr M}\coloneqq \bigsqcup_{x\in G^{(0)}} \ell^\infty(G^x,\mathscr M_x)\subseteq \End(\tilde {\mathscr H})\]
	 	and $\tilde \alpha\coloneqq \Ad(\tilde U)$.
	 One checks that $(G,\tilde{\mathscr M},\tilde \alpha,\mu)$ is a $W^*$-dynamical system satisfying
	 	\begin{equation}\label{eq:L(G,M)}
	 		L^\infty(\tilde{ \mathscr M})= L^\infty(G,r^*\mathscr M,\mu\circ \lambda).
	 	\end{equation}
	 Then $(G,\tilde {\mathscr M},\tilde \alpha,\mu)$ is amenable by the same arguments as in \cite[Ex. 2.1.4(1)]{ADRenault2000amenableGroupoids}.
\end{eg}

Before stating Theorem \ref{thm:amenability for vN algebras} below, let us specify what we mean by \emph{equivariant maps} between $G$-von Neumann algebras. 

\begin{defn}\label{defn:equivariant maps}
	Let $(G,\mathscr M,\alpha,\mu)$ be a $W^*$-dynamical system.
	\begin{enumerate}
		\item For $f\in C_c(G)$ and $m\in L^\infty(\mathscr M)$, define $f*m\in L^\infty(\mathscr M)$ by 
			\[f*m(x)\coloneqq \sum_{g\in G^x}f(g)\alpha_{g}(m(s(g))),\quad x\in G^{(0)}.\]
		\item Let $(G,\mathscr N,\beta,\mu)$ be another $W^*$-dynamical system. We call a linear map $\phi\colon L^\infty(\mathscr M)\to L^\infty(\mathscr N)$ \emph{$G$-equivariant}, if we have 
			\[\phi(f*m)=f*\phi(m),\quad \forall m\in L^\infty(\mathscr M), f\in C_c(G).\]
	\end{enumerate}
\end{defn}

For a $W^*$-dynamical system $(G,\mathscr M,\alpha,\mu)$, we denote by $L^1(G,r^*Z(\mathscr M))^+_1$ the space of measurable sections $f\colon G\to r^*Z(\mathscr M)$ such that $f(g)$ is positive for every $g\in G$ and such that $\sum_{g\in G^x}f(g)\leq 1$ for $\mu$-almost every $x\in G^{(0)}$. 

The following theorem is a generalization of \cite[Thm. 3.3]{anantharaman1987amenableC} (see also \cite[Thm. 3.6]{BeardenCrann}).
\begin{thm}\label{thm:amenability for vN algebras}
	Let $(G,\mathscr M,\alpha,\mu)$ be a $W^*$-dynamical system. The following are equivalent:
\begin{enumerate}
		\item There is a sequence of functions $(f_i)_{i\in \N}$ in $L^1(G,r^*Z(\mathscr M))^+_1$ such that \label{item:f_ifunctions}
			\begin{enumerate}
				\item The functions 
					\[x\mapsto \sum_{g\in G^x}f_i(g)\]
				converge to $1\in L^\infty(Z(\mathscr M))$ ultraweakly.
				\item For every $f\in C_c(G)$ and $\omega\in L^\infty(Z(\mathscr M))_*$, we have
					\[\int_{G^{(0)}}\left|\omega_x\left(\sum_{g,h\in G^x}f(h)\left(f_i(g)-\alpha_h(f_i(h^{-1}g))\right)\right)\right|d\mu(x)\xrightarrow{i\to \infty} 0.\]
					
			\end{enumerate}
		\item $(G,\mathscr M,\alpha,\mu)$ is amenable. \label{item:wAP}
		\item There is a $G$-equivariant conditional expectation 
			\[L^\infty(G,r^*\mathscr M,\mu\circ\lambda)\to L^\infty(\mathscr M).\] \label{item:measureamenable}
		\item There is a $G$-equivariant conditional expectation 
			\[L^\infty(G,r^*Z(\mathscr M),\mu\circ\lambda)\to L^\infty(Z(\mathscr M)).\]
			 \label{item:commmeasureamenable}
	\end{enumerate}
\end{thm}

For the proof, we need the following Lemma:
\begin{lem}\label{lem:commutative G von Neumann algebra}
	Let $(G,\mathscr M,\alpha,\mu)$ be a $W^*$-dynamical system such that $L^\infty(\mathscr M)$ is commutative. Then there exists a second countable locally compact $G$-space $X$, a quasi-invariant measure $\rho$ on $X$ and a $G$-equivariant $*$-isomorphism $L^\infty(\mathscr M)\cong L^\infty(X,\rho)$.
\end{lem}
\begin{proof}[Proof of Lemma \ref{lem:commutative G von Neumann algebra}]
	Let $A\subseteq L^\infty(\mathscr M)$ be a separable ultraweakly dense $C^*$-sub\-al\-ge\-bra satisfying $f*a\in A$ for all $f\in C_c(G)$ and $a\in A$. In particular, we have $f*a\in A$ for every $f\in C_c(G^{(0)})$, so $A$ is a $C_0(G^{(0)})$-algebra. 
	
	For $g\in G$ and $a\in A_{s(g)}$, choose a function $f\in C_c(G)$ supported on an open bisection satisfying $f\equiv 1$ on a neighbourhood of $g$, and a representative $\tilde a\in A$ of $a$. The same arguments as in \cite[Thm. 2.3]{ExactGroupoids} show that $\beta_g(a)\coloneqq (f*\tilde a)(r(g))\in A_{r(g)}$ does not depend on the choices of $f$ and $\tilde a$ and that $g\mapsto \beta_g$ defines an action of $G$ on $A$. Dually, we get an action of $G$ on the Gelfand spectrum $X$ of $A$.
	
	Now fix a faithful normal representation $\Pi\colon L^\infty(\mathscr M)\overline \rtimes G\hookrightarrow \mathcal L(H)$. Use Lemma \ref{lem:williams} to identify $C_0(X)\rtimes_r G$ as a subalgebra of $L^\infty(\mathscr M)\overline \rtimes G$ and write $\pi\coloneqq \Pi|_{C_0(X)}$. 	
	By Propositions 8.7 and 8.19 of \cite{williamsgroupoidbook}, there is a quasi-invariant measure $\rho$ on $X$, a Hilbert bundle $\mathscr H$ on $X$, and an isomorphism $H\cong L^2(X,\mathscr H,\rho)$ that intertwines $\pi$ with the representation of $C_0(X)$ by diagonal operators. 
In particular, the inclusion $C_0(X)\subseteq L^\infty(\mathscr M)$ extends to an isomorphism
		\begin{equation}\label{eq:equiv.iso}
			\overline \pi\colon L^\infty(X,\rho)\cong L^\infty(\mathscr M).
		\end{equation}
	By construction, $\overline \pi$ is equivariant.
\end{proof}

\begin{proof}[Proof of Theorem \ref{thm:amenability for vN algebras}]
	The implication \eqref{item:measureamenable} $\Rightarrow$ \eqref{item:commmeasureamenable} follows from restricting to the center. 
	In view of Lemma \ref{lem:commutative G von Neumann algebra}, the equivalence \eqref{item:commmeasureamenable} $\Leftrightarrow$ \eqref{item:f_ifunctions} is \cite[Prop. 3.1.8]{ADRenault2000amenableGroupoids} (see also \cite[Prop. 10.14]{williamsgroupoidbook}). Similarly, the equivalence \eqref{item:wAP} $\Leftrightarrow$ \eqref{item:commmeasureamenable} is \cite[Prop. 3.1.25]{ADRenault2000amenableGroupoids} (see also \cite[Prop. 10.38]{williamsgroupoidbook}). 
	
	We prove \eqref{item:f_ifunctions} $\Rightarrow$ \eqref{item:measureamenable}, following the proof of \cite[Thm. 3.3]{anantharaman1987amenableC}. For a sequence $(f_i)_{i\in \N}$ as in \eqref{item:f_ifunctions}, we define functions 
		\[P_{f_i}:L^\infty(G,r^*\mathscr M,\mu\circ\lambda)\to L^\infty(\mathscr M),\quad P_{f_i}(\varphi)(x)\coloneqq \sum_{g\in G^x}f_i(x)\varphi(x).\]
	By centrality of the $f_i$, each $P_{f_i}$ is an $L^\infty(\mathscr M)$-bimodule map. Since each $f_i$ belongs to $L^1(G,r^*Z(\mathscr M))_1^+$, each $P_{f_i}$ is completely positive and contractive. By passing to a subnet, we can assume that the $P_{f_i}$ converge to a completely positive contractive $L^\infty(\mathscr M)$-bimodule map $P$ in the topology of pointwise ultraweak convergence. The first condition of \eqref{item:f_ifunctions} implies that $P$ is unital. Thus $P$ is a conditional expectation by $L^\infty(\mathscr M)$-linearity. We claim that $P$ is equivariant. Indeed, let $f\in C_c(G)$ and $\omega\in L^\infty(\mathscr M)_*$ and $\varphi\in L^\infty(G,r^*\mathscr M)$. A straightforward calculation shows that 
		\begin{align*}
		&|\omega(P_{f_i}(f*\varphi)-f*P_{f_i}(\varphi))| \\
		=&\left|\int_{G^{(0)}}\omega_x\left(\sum_{g,h\in G^x}f(h)\alpha_h\left(\varphi(h^{-1}g)\right)\left(f_i(g)-\alpha_h(f_i(h^{-1}g))\right)\right)d\mu(x)\right|\\
		\leq& \|\varphi\|_{\infty} \int_{G^{(0)}}\left|\omega_x\left(\sum_{g,h\in G^x}f(h)\left(f_i(g)-\alpha_h(f_i(h^{-1}g))\right)\right)\right|d\mu(x)\\
		&\xrightarrow{i\to \infty}0.
		\end{align*}
	Since $P_{f_i}\xrightarrow{i\to \infty} P$ pointwise ultraweakly, this implies $P(f*\varphi)=f*P(\varphi)$ as desired. 
\end{proof}

\begin{lem}\label{lem:equivariant E}
	Let $(G,\mathscr M,\alpha,\mu)$ be a $W^*$-dynamical system. Let $U\subseteq G$ be an open bisection and denote by $v_U\in L^\infty(\mathscr M)\overline \rtimes G$ the partial isometry given by the characteristic function of $U$. Denote by $L^\infty(\mathscr M)\xrightarrow{\iota}L^\infty(\mathscr M)\overline \rtimes G\xrightarrow{E}L^\infty(\mathscr M)$ the canonical inclusion and conditional expectation. 
	\begin{enumerate}
		\item For every $\varphi\in C_c(U)$ and $f\in L^\infty(\mathscr M)$, we have 
			\begin{equation*}
				\iota(\varphi * f) = \iota(\varphi\circ r|_U^{-1})\cdot v_U \iota(f) v_U^*.
			\end{equation*}\label{item: action by conjugation}
		\item For every $x\in L^\infty(\mathscr M)\overline \rtimes G$, we have 
			\begin{equation*}
				\iota\circ E(v_Uxv_U^*)=v_U(\iota\circ E(x))v_U^*.
			\end{equation*}\label{item: equivariant E}
	\end{enumerate}
\end{lem}
\begin{proof}
	For the first part of the Lemma view $\iota(\varphi*f)$ as an element in $S(r^*\mathscr M)$. We have
		\[\iota(\varphi * f)(g)=\begin{cases}	0,&g\notin r(U)\\ \varphi(h)\alpha_h(f(s(h))),&U\cap r^{-1}(\{g\})=\{h\} \end{cases}.\]
	Also viewing $v_U\iota(f)v_U^*$ as an element in $S(r^*\mathscr M)$, we have
		\begin{align*}
		v_U \iota(f)v_U^*(g)=&\sum_{h\in G^{r(g)}}(v_U\iota(f))(h)\alpha_g(1_U(g^{-1}h))\\
		=&\sum_{h,k\in G^{r(g)}}1_U(k)\alpha_k(f(k^{-1}h))\alpha_g(1_U(g^{-1}h))\\
		=&\begin{cases} 0, &g\notin r(U)\\ \alpha_h(f(s(h))),&U\cap r^{-1}(\{g\})=\{h\}\end{cases}.
		\end{align*}
		This proves the first part of the Lemma.	
	For the second part, we may assume $x\in S(r^*\mathscr M)$ by normality of $\iota$ and $E$. For $g\in G$, we have  
		\begin{equation}\label{eq:conjugate with U}
			v_Uxv_U^*(g)=\sum_{h,k\in G^{r(g)}}1_U(k)\alpha_k(x(k^{-1}h))\alpha_g(1_U(g^{-1}h)).
		\end{equation}
	From this it follows that 
		\[\iota\circ E(v_Uxv_U^*)(g)=\begin{cases} 0, &g\notin r(U)\\\alpha_h(x(s(h))),&U\cap r^{-1}(\{g\})=\{h\} \end{cases}.\]
	If we replace $x$ by $\iota\circ E(x)$ in \eqref{eq:conjugate with U}, only the terms with $h=k\in U$ remain and we obtain  
		\[v_U (\iota\circ E(x))v_U^*(g)=\begin{cases} 0, &g\notin r(U) \\ \alpha_h(x(s(h)),&U\cap r^{-1}(\{g\})=\{h\} \end{cases}.\]
\end{proof}

\begin{prop}\label{prop:injectivity of crossed products}
	Let $(G,\mathscr M,\alpha,\mu)$ be a $W^*$-dynamical system. If $L^\infty(\mathscr M)\overline{\rtimes}G$ is injective then $L^\infty(\mathscr M)$ is injective and $(G,\mathscr M,\alpha,\mu)$ is amenable.
\end{prop}
\begin{proof}
	We write $\nu\coloneqq \mu\circ \lambda$.
	Injectivity of $L^\infty(\mathscr M)$ follows from existence of the conditional expectation $E\colon L^\infty(\mathscr M)\overline{\rtimes}G\to L^\infty(\mathscr M)$. We consider $L^\infty(\mathscr M)\subseteq L^\infty(G,r^*\mathscr M,\nu)$ as a subalgebra and similarly $L^\infty(\mathscr M)\overline{\rtimes}G\subseteq L^\infty(G,r^*\mathscr M,\nu)\overline \rtimes G$ as a subalgebra via pulling back along the range map. Since $L^\infty(\mathscr M)\overline \rtimes G$ is injective, there is a completely positive map 
		\[Q\colon L^\infty(G,r^*\mathscr M,\nu)\overline \rtimes G\to L^\infty(\mathscr M)\overline \rtimes G\]
	extending the identity on $L^\infty(\mathscr M)\overline \rtimes G$. By composing with the inclusion \[\iota\colon L^\infty(G,r^*\mathscr M,\nu)\hookrightarrow L^\infty(G,r^*\mathscr M,\nu)\overline \rtimes G\]
	 and the conditional expectation $E\colon L^\infty(\mathscr M)\overline \rtimes G\to L^\infty(\mathscr M)$, we get a conditional expectation
		\[P\coloneqq E\circ Q\circ \iota\colon L^\infty(G,r^*\mathscr M,\nu)\to L^\infty(\mathscr M).\]
	We show that $P$ is equivariant, that is $P(\varphi*f)=\varphi*P(f)$ for all $\varphi\in C_c(G)$ and $f\in L^\infty(G,r^*\mathscr M,\nu)$. By a partition of unity argument, we may assume that $\varphi$ is supported on an open bisection $U\subseteq G$. Using Lemma \ref{lem:equivariant E}\eqref{item: action by conjugation} and the fact that both $v_U$ and $\iota(\varphi\circ r|_U^{-1})$ lie in the multiplicative domain of $Q$, we get 
		\[P(\varphi * f)= E( Q(\iota(\varphi\circ r|_U^{-1})\cdot v_U \iota(f) v_U^*))= (\varphi\circ r|_U^{-1})E( v_U Q(\iota(f)) v_U^*)\]
		and therefore 
		\[\iota\circ P(\varphi * f)=\iota(\varphi\circ r|_U^{-1})\iota\circ E( v_U Q(\iota(f)) v_U^*) =\iota(\varphi\circ r|_U^{-1}) v_U\iota\circ P(f)v_U^*= \iota(\varphi *P(f)),\]
		where we have used Lemma \ref{lem:equivariant E}\eqref{item: equivariant E} at the second equality and Lemma \ref{lem:equivariant E}\eqref{item: action by conjugation} at the third equality.
		Since $\iota$ is injective, this proves $P(\varphi * f)=\varphi * P(f)$. Now amenability of $(G,\mathscr M,\alpha,\mu)$ follows from Theorem \ref{thm:amenability for vN algebras}.
\end{proof}

\subsection{Measurewise amenable actions on $C^*$-algebras}

\begin{defn}\label{defn: amenable G C^*}
\begin{enumerate}
	\item A $C^*$-dynamical system $(G,\mathscr A,\alpha)$ is called \emph{measurewise amenable} if for any of its covariant representations $(\mu,\mathscr H,\pi,U)$, the $W^*$-dynamical system $(G,\pi(\mathscr A)'',\alpha'',\mu)$ introduced in Example \ref{eg:enveloping G von Neumann algebra} is amenable in the sense of Definition \ref{defn:measurewise amenable}. 
	\item We say that $(G,\mathscr A,\alpha)$ has the \emph{approximation property}, if the associated semi-direct product Fell bundle $\mathscr A\rtimes_\alpha G$ defined in Definition \ref{defn:semidirect product} has the approximation property in the sense of Definition \ref{defn:AP}. 
\end{enumerate}
We say that a $G$-$C^*$-algebra $(A,\alpha)$ is measurewise amenable or has the approximation property if the associated $C^*$-dynamical system does so. 
\end{defn}

\begin{eg} Let $(G,\mathscr A,\alpha)$ be a $C^*$-dynamical system.
	\begin{enumerate}
		\item In the case that $C_0(\mathscr A)$ is commutative, the action $\alpha$ is induced by an action of $G$ on the Gelfand Spectrum $X$ of $C_0(\mathscr A)$. Then the system $(G,\mathscr A,\alpha)$ is measurewise amenable if and only if the transformation groupoid $X\rtimes G$ is measurewise amenable in the sense of \cite{ADRenault2000amenableGroupoids} and it has the approximation property if and only if $X\rtimes G$ is topologically amenable in the sense of \cite{ADRenault2000amenableGroupoids}. Since $G$ is \'etale, both notions of amenability are equivalent by \cite[Rem. 3.3.9]{ADRenault2000amenableGroupoids}. 
		\item In the case that $G$ is a discrete group, the system $(G,\mathscr A,\alpha)$ is measurewise amenable if and only if the action $\alpha\colon G\acton  C_0(\mathscr A)$ is amenable in the sense of \cite{anantharaman1987amenableC} and it has the approximation property if and only if the associated Fell bundle has the approximation property in the sense of \cite{Exel-Crelle}. In this setting, both notions of amenabilty are equivalent by \cite[Thm. 2.13, Thm. 3.2]{OzawaSuzuki}.
	\end{enumerate}
\end{eg}

The above examples suggest that measurewise amenability and the approximation property might be equivalent for all $C^*$-dynamical systems satisfying our standing assumptions. Although we are not able to prove this equivalence in general, we have the following:

\begin{thm}\label{thm:AP implies measurewise amenable}
	Let $(G,\mathscr A,\alpha)$ be a $C^*$-dynamical system with the approximation property. Then $(G,\mathscr A,\alpha)$ is measurewise amenable. 
\end{thm}
\setcounter{claim}{0}
\begin{proof}
	The proof follows the main ideas of \cite[Lem. 6.5]{abadie2021amenability}. Let $(\mu,\mathscr H,\pi,U)$ be a covariant representation of $(G,\mathscr A,\alpha)$. Write $A=C_0(\mathscr A)$ and $\nu=\mu\circ\lambda$. By Theorem \ref{thm:amenability for vN algebras} it suffices to construct a $G$-equivariant conditional representation
		\[Q\colon L^\infty(G,r^*Z(\pi(\mathscr A)''),\nu)\to Z(\pi(A)'').\]
	Let $(e_\lambda)_\lambda$ be an approximate unit for $A$ in $C_c(G^{(0)},\mathscr A)$. Let $(a_i)_{i\in \N}$ be a sequence in $C_c(G,r^*\mathscr A)$ witnessing the approximation property. We consider $L^2(G,r^*\pi(\mathscr A)'')$ as a Hilbert module over $L^\infty(G^{(0)},\pi(\mathscr A)'',\mu)$ as in the first paragraph of Section \ref{sec:amenable vN}. Then $L^\infty(G,r^*\pi(\mathscr A)'',\nu )$ acts on $L^2(G,r^*\pi(\mathscr A)'')$ by multiplication operators. We consider $\pi( A)''\subseteq L^\infty(G,\pi(\mathscr A)'',\nu )$ as a subalgebra by pulling back along the range map $r$. In particular, we may act on $L^2(G,r^*\pi(\mathscr A)'')$ by $\left(e_\lambda^\frac 1 2\right)_\lambda$. We define a net of completely positive maps $(P_{i,\lambda})_{i,\lambda}$ by 
	\[P_{i,\lambda}\colon L^\infty(G,\pi(\mathscr A)'',\nu )\to \pi(A)'',\quad f\mapsto \left\langle e_\lambda^{\frac 1 2}a_i,fe_\lambda^{\frac 1 2}a_i\right\rangle.\]
	Note that all the $P_{i,\lambda}$ are contractive by the first condition of Definition \ref{defn:AP}. For a fixed $i\in \N$, we consider the product map 
	\[\prod_\lambda P_{i,\lambda}\colon L^\infty(G,r^*\pi(\mathscr A)'',\nu )\to \prod_\lambda \pi(A)''\]
	and we fix a limit point $\tilde P=\prod_\lambda P_\lambda$ of the sequence $\left(\prod_\lambda P_{i,\lambda}\right)_{i\in \N}$		
	in the topology of pointwise ultraweak convergence. We denote by 
		\[Q_\lambda\colon L^\infty(G,r^*Z(\pi(\mathscr A)''),\nu )\to \pi(A)''\]
	the restriction of $P_\lambda$ to the center.
	
	\begin{claim}\label{clm:Q converges}
		The net $(Q_\lambda)_\lambda$ converges to a completely positive map $Q$ in the topology of pointwise ultraweak convergence. 
	\end{claim}	
	Since the maps $(Q_\lambda)_\lambda$ are completely positive and uniformly bounded, it suffices to show that they are increasing. Fix a positive element $f\in L^\infty(G,r^*Z(\pi(\mathscr A)''),\nu )$ and $\lambda\leq \lambda'$. Then for every $i\in \N$, we have 
		\[\left\langle e_\lambda^{\frac 1 2}a_i,fe_\lambda^{\frac 1 2}a_i\right\rangle = \left\langle a_i,e_\lambda f a_i\right\rangle \leq \left\langle a_i,e_{\lambda'} f a_i\right\rangle =\left\langle e_{\lambda'}^{\frac 1 2}a_i,fe_{\lambda'}^{\frac 1 2}a_i\right\rangle\]
	by centrality and positivity of $f$. Taking the limit along $i\to \infty$, we obtain $Q_\lambda(f)\leq Q_{\lambda'}(f)$ as desired. 
	
	\begin{claim}\label{clm:Q is central}
		The image of $Q$ is contained in $Z(\pi(A)'')$. 
	\end{claim}
	To prove this, it suffices to prove that for every positive $f\in L^\infty(G,r^*Z(\pi(\mathscr A)''),\nu )$ and every self-adjoint $b\in \pi(A)''$, the element $Q(f)b\in \pi(A)''$ is self-adjoint. It moreover suffices to consider $b\in \pi(A)$. By passing to a subnet, we may assume that the net $(P_\lambda)_\lambda$ converges pointwise ultraweakly to a completely positive map $P\colon L^\infty(G,r^*\pi(\mathscr A)'',\nu )\to \pi(A)''$. Since $P(fb)$ is self-adjoint, it suffices to show that $Q(f)b=P(fb)$. For a normal positive linear functional $\omega\in \pi(A)''_*$, we denote by $\|\cdot \|_\omega$ the norm on $L^2(G,r^*\pi(\mathscr A)'')$ given by $\|\xi\|_\omega\coloneqq\sqrt{\omega(\langle \xi,\xi\rangle)}$. 
We have

\begin{align*}
	|\omega(Q(f)b-P(fb))|=&\lim_\lambda\lim_i \left|\omega\left(\left\langle e_\lambda^{\frac 1 2}a_i,fe_\lambda^{\frac 1 2}a_i\right\rangle b-\left\langle e_\lambda^{\frac 1 2}a_i,fbe_\lambda^{\frac 1 2}a_i\right\rangle\right)\right|\\
	=&\lim_\lambda\lim_i\left|\omega\left(\left\langle fe_\lambda^{\frac 1 2}a_i,e_\lambda^{\frac 1 2}a_ib-be_\lambda^{\frac 1 2}a_i\right\rangle \right)\right| \\
	\leq &\lim_\lambda \lim_i \left \|fe_\lambda^{\frac 1 2} a_i\right\| \left\|e_\lambda^{\frac 1 2}a_ib-be_\lambda^{\frac 1 2} a_i\right\|_\omega\\
	\leq &\|f\| \lim_\lambda\lim_i \left\|e_\lambda^{\frac 1 2}a_ib-be_\lambda^{\frac 1 2} a_i\right\|_\omega
\end{align*}
where at the second to last inequality we have used the Cauchy-Schwarz inequality and at the last inequality we have used that $\left\| e_\lambda ^{\frac 1 2}a_i\right\|\leq 1$. The last term in the above inequality is zero since we have 
	\begin{align*}
		&\lim_\lambda\lim_i \omega\left(\left\langle e_\lambda^{\frac 1 2}a_i b-b e_\lambda^{\frac 1 2}a_i,e_\lambda^{\frac 1 2}a_ib-be_\lambda^{\frac 1 2}a_i\right\rangle \right)\\
		=&\lim_\lambda\lim_i \omega\left(\left\langle e_\lambda^{\frac 1 2}a_ib,e_\lambda^{\frac 1 2}a_ib\right\rangle-\left\langle e_\lambda^{\frac 1 2}a_ib,be_\lambda^{\frac 1 2}a_i\right\rangle -\left\langle be_\lambda^{\frac 1 2}a_i,e_\lambda^{\frac 1 2}a_ib\right\rangle+\left\langle be_\lambda^{\frac 1 2}a_i,be_\lambda^{\frac 1 2}a_i\right\rangle\right)\\
		=&\lim_\lambda\lim_i \omega\left(b\left\langle a_i,e_\lambda a_i\right\rangle b	-b\left\langle a_i,e_\lambda^{\frac 1 2}be_\lambda^{\frac 1 2}a_i\right\rangle -\left\langle a_i,e_\lambda^{\frac 1 2}be_\lambda ^{\frac 1 2}a_i\right\rangle b+\left\langle a_i,e_\lambda^{\frac 1 2}b^2e_\lambda^{\frac 1 2}a_i\right\rangle \right)\\
		=&\lim_\lambda \omega\left(be_\lambda b-be_\lambda^{\frac 1 2}be_\lambda^{\frac 1 2}-e_\lambda ^{\frac 1 2}be_\lambda^{\frac 1 2}b+e_\lambda^{\frac 1 2}b^2e_\lambda^{\frac 1 2}\right) \\
		=&0.
	\end{align*}
	Here we used Condition \eqref{item:W* almost inv} of Definition \ref{defn:measurewise amenable} in the second to last step. Since $\omega$ was arbitrary, this proves $Q(f)b=P(fb)$ and establishes Claim \ref{clm:Q is central}. 
	
	\begin{claim}\label{clm:Q conditional expectation}
		$Q$ is a conditional expectation.
	\end{claim}	
	For $b\in Z(\pi(A)'')$, we have 
	\begin{align*}
		Q(b)=&\lim_\lambda \lim_i \left\langle e_\lambda^{\frac 1 2}a_i,be_\lambda^{\frac 1 2}a_i\right\rangle =\lim_\lambda\lim_i\langle a_i,e_\lambda a_ib\rangle\\
		=&\lim_\lambda\lim_i\langle a_i,e_\lambda a_i\rangle b =\lim_\lambda e_\lambda b=b,
	\end{align*}
	where we have used $b\in Z(\pi(A)'')$ at the second equality.
	\begin{claim}
		$Q$ is equivariant. 
	\end{claim}	
	As in Example \ref{eg:Linfty(G,M)}, denote the diagonal $G$-action on $L^\infty(G,r^*\pi(\mathscr A)'',\nu )$ by $\tilde \alpha$ and denote the diagonal $G$-action on $L^2(G,r^*\pi(A)'')$ by $V$. Note that $\tilde \alpha$ is implemented by conjugation with $V$ and that we have $\langle V_g\xi,V_g\zeta\rangle=\alpha_g(\langle\xi,\zeta\rangle)$.

Let $f\in C_c(G)$ and $\varphi\in L^{\infty}(G,Z(r^*\pi(\mathscr A)''),\nu )$ be positive elements. Let $\omega\in \pi(A)''_*$ be a normal state and write $\omega=\int_{G^{(0)}}\omega_xd\mu(x)$ as in Lemma \ref{lem:normal states on measured fields}. Using the inequality 
	\begin{equation}\label{eq:elementary norm estimate}	
	|\|x\|^2-\|y\|^2|\leq (\|x\|+\|y\|)\|x-y\|
	\end{equation}
	(which holds for any two elements $x,y$ in a normed vector space), we get the estimate 

\begin{align*}
	&|\omega(f*P_{i,\lambda}(\varphi)-P_{i,\lambda}(f*\varphi))|\\
	=&\left|\int_{G^{(0)}}\omega_x\left((f*P_{i,\lambda}(\varphi))(x)-P_{i,\lambda}(f*\varphi)(x)\right)d\mu(x)\right| \\
	=&\left|
		\int_{G^{(0)}}\omega_x
		\left(
			\sum_{g\in G^x}f(g)
			\left(
				\alpha_g
				\left(
					\left\langle e_\lambda^{\frac 1 2}(s(g))a_i|_{G^{s(g)}},\varphi|_{G^{s(g)}}e_\lambda^{\frac 1 2}(s(g))a_i|_{G^{s(g)}}
					\right\rangle
				\right)
			\right.
		\right.
	\right. \\
	&\left.
		\left.
			\left.
			-\left\langle e_\lambda^{\frac 1 2}(x)a_i|_{G^x},\tilde \alpha_g(\varphi|_{G^{s(g)}})e_\lambda^{\frac 1 2}(x)a_i|_{G^x}\right\rangle
			\right)
		\right)d\mu(x)\right|\\
	=&\left|
		\int_{G^{(0)}}\omega_x
		\left(
			\sum_{g \in G^x}f(g)
			\left(
					\left\langle \tilde \alpha_g(\varphi|_{G^{s(g)}})^{\frac 1 2}V_g \left(e_\lambda^{\frac 1 2}(s(g))a_i|_{G^{s(g)}}\right),
					\right.
			\right.
		\right.
	\right.\\
	&\left.
		\left.
			\left.
					\left.					
					\tilde \alpha_g(\varphi|_{G^{s(g)}})^{\frac 1 2}V_g\left(e_\lambda^{\frac 1 2}(s(g))a_i|_{G^{s(g)}}\right)
					\right\rangle
			\right.
		\right.
	\right. \\
	&\left.
		\left.
			\left.
			-\left\langle \tilde \alpha_g(\varphi|_{G^{s(g)}})^{\frac 1 2}e_\lambda^{\frac 1 2}(x)a_i|_{G^x},\tilde  \alpha_g(\varphi|_{G^{s(g)}})^{\frac 1 2}e_\lambda^{\frac 1 2}(x)a_i|_{G^x}\right\rangle
			\right)
		\right)d\mu(x)\right|\\ 
		\leq &\int_G f(g)\left| \left\|\tilde \alpha_g(\varphi|_{G^{s(g)}})^{\frac 1 2}V_g \left(e_\lambda^{\frac 1 2}(s(g))a_i|_{G^{s(g)}}\right)\right\|_{\omega_{r(g)}}^2-\left\|\tilde \alpha_g(\varphi|_{G^{s(g)}})^{\frac 1 2}e_\lambda^{\frac 1 2}(x)a_i|_{G^x}\right\|_{\omega_{r(g)}}^2\right|d\nu(g) \\
		\leq&D\cdot \int_G f(g)\left\|e^{\frac 1 2}_\lambda(r(g))a_i|_{G^{r(g)}}-V_g\left(e_\lambda^{\frac 1 2}(s(g))a_i|_{G^{s(g)}}\right)\right\|_{\omega_{r(g)}}d\nu (g)
\end{align*}
for some constant $D\in (0,\infty)$ that as a consequence of the first condition of Definition \ref{defn:AP} only depends on $\varphi$, but not on $i$ or on $\lambda$. Here, we have used \eqref{eq:elementary norm estimate} at the last inequality. We rewrite the square of the integrand of the above integral as
\begin{align*}
	&\left(f(g)\left\|e^{\frac 1 2}_\lambda(r(g))a_i|_{G^{r(g)}}-V_g\left(e_\lambda^{\frac 1 2}(s(g))a_i|_{G^{s(g)}}\right)\right\|_{\omega_{r(g)}}\right)^2\\
	=&f(g)^2\omega_{r(g)}\Bigg( 
		\left\langle e^{\frac 1 2}_\lambda(r(g))a_i|_{G^{r(g)}},e^{\frac 1 2}_\lambda(r(g))a_i|_{G^{r(g)}}\right\rangle\\
			-&\left\langle e^{\frac 1 2}_\lambda(r(g))a_i|_{G^{r(g)}}, \alpha_g(e_\lambda^{\frac 1 2}(s(g)))V_g\left(a_i|_{G^{s(g)}}\right)	\right\rangle\\
			-&\left \langle   \alpha_g(e_\lambda^{\frac 1 2}(s(g)))V_g\left(a_i|_{G^{s(g)}}\right) ,e^{\frac 1 2}_\lambda(r(g))a_i|_{G^{r(g)}}\right\rangle \\
			+&\left\langle  \alpha_g(e_\lambda^{\frac 1 2}(s(g)))V_g\left(a_i|_{G^{s(g)}}\right), \alpha_g(e_\lambda^{\frac 1 2}(s(g)))V_g\left(a_i|_{G^{s(g)}}\right)\right\rangle
		\Bigg)
\end{align*}

By the second condition of Definition \ref{defn:AP}, this expression converges for $i\to \infty$ to
\begin{align*}
		f(g)^2\omega_{r(g)}\left(e_\lambda (r(g))-e_\lambda^{\frac 1 2}(r(g)) \alpha_g(e_\lambda^{\frac 1 2}(s(g)))- \alpha_g(e_\lambda^{\frac 1 2}(s(g)))e_\lambda^{\frac 1 2}(r(g))- \alpha_g(e_\lambda(s(g)))\right),
\end{align*}

where the convergence is uniformly in $g$.
Since $(e_\lambda)_\lambda$ is an approximate unit, the above term converges to $0$ for all $g\in G$\footnote{If we take $(e_\lambda)_\lambda$ to be the set of all positive contractions in $C_c(\mathscr A)$, then $(e_\lambda^{\frac 1 2})_\lambda$ is a subnet and therefore an approximate unit itself. Then $(e_\lambda^{\frac 1 2}(r(g)))_\lambda$ and $(\alpha_g(e_\lambda^{\frac 1 2}(s(g)))_\lambda$ are both approximate units for $\mathscr A_{r(g)}$, so they converge to $1$ ultraweakly and therefore the expression displayed converges to zero.}. The above calculations together with Lebesgue's dominated convergence theorem imply that 
	\[|\omega(f*Q(\varphi)-Q(f*\varphi))|=\lim_\lambda\lim_i |\omega(f*P_{i,\lambda}(\varphi)-P_{i,\lambda}(f*\varphi))|=0.\]
Since $\omega\in \pi(A)''_*$ was arbitrary, this implies $f*Q(\varphi)=Q(f*\varphi)$ and finishes the proof. 
\end{proof}

We close this section by showing that measurewise amenability implies equality of the full and reduced crossed products:

\begin{thm}\label{thm:measurewise amenable implies WC}
	Let $(G,\mathscr A,\alpha)$ be a measurewise amenable $C^*$-dynamical system. Then we have $C_0(\mathscr A)\rtimes G=C_0(\mathscr A)\rtimes_r G$. 
\end{thm}
\begin{proof}
	The proof uses the same idea as the proof of Theorem \ref{thm:APimpliesWC} and is modeled after \cite[Thm. 1]{williamsamenability}. 
	We show that the integrated form $\pi\rtimes U$ of  every covariant representation $(\mu,\mathscr H,\pi,U)$ of $(G,\mathscr A,\alpha)$ factors through the reduced crossed product. Let $(a_i)_{i\in \N}$ be a sequence in $L^1(G,r^*Z(\pi(\mathscr A)''))_1^+$ as in Definition \ref{defn:measurewise amenable}, $\nu\coloneqq\mu\circ\lambda $ and $\Delta\coloneqq \frac{d\nu}{d\nu^{-1}}$. We define operators
	\[T_i\colon L^2(G^{(0)},\mathscr H,\mu)\to L^2(G,r^*\mathscr H,\nu^{-1}),\quad T_i\xi(g)\coloneqq\Delta^{\frac 1 2}(g)a_i(g)\xi(r(g))\]
	whose adjoints are given by
		\[T_i^*\eta(x)=\sum_{g\in G^x}a_i(g)^*\eta(g)\Delta^{-\frac 1 2}(g),\quad \eta\in L^2(G^{(0)},\mathscr H,\mu), x\in G^{(0)}.\]
	Then the maps
		\[V_i\colon C_c(G,r^*\mathscr A)\to \mathcal L(L^2(G^{(0)},\mathscr H,\mu)),\quad V_i(f)\coloneqq T_i^*M_\pi(f)T_i\]
	extend to $C_0(\mathscr A)\rtimes_r G$ by Lemma \ref{lem:williams} where $M_\pi$ is given by 
		\[M_\pi(f)\xi(g)=\sum_{h\in G^{r(g)}}\pi(f(h)) U_h\xi(h^{-1}g)\]
	for $f\in C_c(G,r^*\mathscr A), \xi \in L^2(G^{(0)},\mathscr H,\mu)$, and $g\in G$.
	It follows from the first condition of Definition \ref{defn:measurewise amenable} that the maps $(V_i)_{i\in \N}$ are uniformly bounded. Now for $f\in C_c(G,r^*\mathscr A)$ and $\xi,\eta\in L^2(G^{(0)},\mathscr H,\mu)$, we have
	\begin{align*}
		&\langle \xi,V_i(f)\eta\rangle =\int_{G^{(0)}}\langle \xi(x)(T_i^*M_\pi(f)T_i\eta)(x)\rangle d\mu(x)\\
		=&\int_{G^{(0)}}\left\langle \xi(x),\sum_{g\in G^x}a_i(g)^*\left((M_\pi(f)T_i\eta)(g)\right)\Delta^{-\frac 1 2}(g) \right\rangle d\mu(x)\\
		=&\int_{G^{(0)}}\left\langle \xi(x),\sum_{g\in G^x}a_i(g)^*\left(\sum_{h\in G^x}\pi(f(h))U_h((T_i \eta)(h^{-1}g))\right)\Delta^{-\frac 1 2}(g)\right\rangle d\mu(x)\\
		=&\int_{G^{(0)}}\sum_{g,h\in G^x}\left\langle \xi(x),a_i(g)^*\pi(f(h)U_ha_i(h^{-1}g)\eta(s(h))\Delta^{\frac 1 2}(h^{-1}g)\Delta^{-\frac 1 2}(g)\right\rangle d\mu(x)\\
		=&\int_{G^{(0)}}\sum_{h\in G^x}\left\langle\xi(x),\underbrace{\left(\sum_{g\in G^x}a_i(g)^*\alpha_h(a_i(h^{-1}g))\right)}_{\eqqcolon\tilde a_i(h)}\pi(f(h))U_h(\eta(s(h)))\Delta^{-\frac 1 2}(h)\right\rangle d\mu(x)\\
		=&\int_G\left\langle \xi(r(h)),\tilde a_i(h)\pi(f(h))U_h(\eta(s(h)))\Delta^{-\frac 1 2}(h)\right\rangle d\nu(h)\\
		\xrightarrow{i\to \infty} &\int_G\left\langle \xi(r(h)),\pi(f(h))U_h(\eta(s(h)))\Delta^{-\frac 1 2}(h)\right\rangle d\nu(h)=\langle \xi,\pi\rtimes U(f)\eta\rangle
	\end{align*}
	Here we used centrality of the $(a_i)_{i\in \N}$, the second condition of Definition \ref{defn:measurewise amenable}, and Lemma \ref{lem:hom-almost-everywhere}. Since the maps $(V_i)_{i\in \N}$ are uniformly bounded, an $\frac\varepsilon 3$-argument shows that we also have $\langle \xi,V_i(f)\eta\rangle \xrightarrow{i\to \infty} \langle \xi,\pi\rtimes U(f)\eta\rangle$ for general $f\in C_0(\mathscr A)\rtimes G$. Since all the $(V_i)_{i\in \N}$ factor through $C_0(\mathscr A)\rtimes_r G$, it follows that $\pi\rtimes U$ factors through $C_0(\mathscr A)\rtimes_r G$. 
\end{proof}

\subsection{Nuclearity of crossed products}
The following theorem is a generalization of \cite[Thm. 4.5]{anantharaman1987amenableC}.
\begin{thm}\label{thm:Nuclearity for groupoid crossed products}
	Let $(G,\mathscr A,\alpha)$ be a $C^*$-dynamical system. Then $C_0(\mathscr A)\rtimes_r G$ is nuclear if and only if $C_0(\mathscr A)$ is nuclear and $(G,\mathscr A,\alpha)$ is measurewise amenable. 
\end{thm}

For the proof we need the following lemma.

\begin{lem}\label{lem: measurewise amenable and tensor products}
	Let $(G,\mathscr A,\alpha)$ be a measurewise amenable $C^*$-dynamical system and  $B$ a separable, unital $C^*$-algebra. Then $(G,\mathscr A\otimes_{\max}B,\alpha\otimes \id)$ is measurewise amenable too. If we additionally assume that $C_0(\mathscr A)$ is exact, then the same statement holds for the minimal tensor product. 
\end{lem}

\begin{proof}
	We prove the statement for the maximal tensor product only since the proof for the minimal tensor product is identical (we only use exactness in the construction of $\mathscr A\otimes B$).
	Let $(\mu,\mathscr H,\pi,U)$ be a covariant representation of $(G,\mathscr A\otimes_{\max}B,\alpha\otimes \id)$. By \cite[Lem.~4.3]{LalondeNuclearity}, there are non-degenerate commuting representations 
	\[\pi_\mathscr A\colon C_0(\mathscr A)\to \mathcal L(L^2(G^{(0)},\mathscr H,\mu)),\quad \pi_B\colon B\to \mathcal L(L^2(G^{(0)},\mathscr H,\mu))\]
	satisfying $\pi(a\otimes b)=\pi_\mathscr A(a)\pi_B(b)$ for all $a\in C_0(\mathscr A),b\in B$ and such that $(\mu,\mathscr H,\pi_\mathscr A,U)$ is a covariant representation of $(G,\mathscr A,\alpha)$. Denote by 
		\[\iota\colon Z(\pi_\mathscr A(\mathscr A)'')\hookrightarrow Z(\pi(\mathscr A\otimes_{\max}B)'')\subseteq \End(\mathscr H)\]
	the canonical inclusion. Note that $\iota$ is unital and normal. One easily checks that if $(a_i)_{i\in \N}$ is a sequence of functions witnessing amenability of the $W^*$-dynamical system $(G,\pi_\mathscr A(\mathscr A)'',\alpha'',\mu)$, then $(\iota\circ a_i)_{i\in \N}$ is a sequence of functions witnessing amenability of the $W^*$-dynamical system $(G,\pi(\mathscr A\otimes_{\max}B)'',\alpha''\otimes \id,\mu)$. This shows that $(G,\mathscr A\otimes_{\max}B,\alpha)$ is measurewise amenable.
\end{proof}

\begin{proof}[Proof of Theorem \ref{thm:Nuclearity for groupoid crossed products}]
	Suppose first that $C_0(\mathscr A)\rtimes_r G$ is nuclear. Then $C_0(\mathscr A)$ is nuclear since there is a conditional expectation $E\colon C_0(\mathscr A)\rtimes_r G\to C_0(\mathscr A)$. Let $(\mu,\mathscr H,\pi,U)$ be a covariant representation of $(G,\mathscr A,\alpha)$. Using Lemma \ref{lem:williams}, it is easy to see that the von Neumann crossed product $L^\infty(\pi(\mathscr A)'')\overline{\rtimes}G$ is the weak closure of (possibly a quotient of) the reduced crossed product $C_0(\mathscr A)\rtimes_r G$. Since $C_0(\mathscr A)\rtimes_r G$ is nuclear, $L^\infty(\pi(\mathscr A)'')\overline \rtimes G$ is semi-discrete and therefore injective. It follows from Proposition \ref{prop:injectivity of crossed products} that $(G,\pi(\mathscr A)'',\alpha'',\mu)$ is amenable. Since $(\mu,\mathscr H,\pi,U)$ was arbitrary, we conclude that $(G,\mathscr A,\alpha)$ is measurewise amenable. 
	
	Suppose now that $C_0(\mathscr A)$ is nuclear and that $(G,\mathscr A,\alpha)$ is measurewise amenable. Then $C_0(\mathscr A)\rtimes_r G$ is nuclear by the same proof as in Theorem \ref{thm:APimpliesnuclearity} using Lemma \ref{lem: measurewise amenable and tensor products} instead of Lemma \ref{lem:APandtensorproducts} and Theorem \ref{thm:measurewise amenable implies WC} instead of Theorem \ref{thm:APimpliesWC}.
\end{proof}

\section{The non-\'etale case}

We briefly point out what parts of the paper carry over if we consider a second countable, locally compact, Hausdorff groupoid $G$ that is not necessarily \'etale. One modification that applies to all sections is that direct sums indexed over fibers $G^x$ of $G$ have to be replaced by integrals with respect to a Haar system. The following parts of each section can be adapted to the non-\'etale case:

\subsection*{Preliminaries}
The definitions of Fell bundles, $C^*$-dynamical systems and $W^*$-dynamical systems are identical. The definition of the associated $C^*$-algebras has to be modified to only take into account representations that are bounded in \emph{$\|\cdot \|_I$-norm} (see \cite{Muhlydisintegration,Renault1987Representations,OtyRamsay2006GvNA}). The disintegration theorems still hold. However, the canonical conditional expectations no longer exist. Although we expect Proposition \ref{prop:algebraic Fell bundles} to still hold, we are not able to prove it (the argument that we borrow from \cite{takeishi} uses that $G$ is \'etale).

\subsection*{An approximation property for Fell bundles}
Definition \ref{defn:AP} can still be stated and Theorem \ref{thm:APimpliesWC} still holds with the same proof. We expect Theorem \ref{thm:APimpliesnuclearity} to still hold, but we are currently not able to remove the \'etaleness assumption in the proofs of Proposition \ref{prop:algebraic Fell bundles} (needed for the construction of tensor products) and Lemma \ref{lem:maximalandmaximalproduct}. 

\subsection*{Amenability for groupoid actions}
Definitions \ref{defn:measurewise amenable} and \ref{defn: amenable G C^*} can still be stated and Theorems \ref{thm:AP implies measurewise amenable} and \ref{thm:measurewise amenable implies WC} hold with the same proofs. We expect Theorem \ref{thm:amenability for vN algebras} to still hold, but its proof relies on Lemma \ref{lem:commutative G von Neumann algebra}, which we can currently only prove for \'etale groupoids. Proposition \ref{prop:injectivity of crossed products} is false in general: The group von Neumann algebra of $\SL(2,\R)$ for example is injective while $\SL(2,\R)$ is not amenable. For the same reason the "only if" implication of Theorem \ref{thm:Nuclearity for groupoid crossed products} fails. The "if" implication of Theorem \ref{thm:Nuclearity for groupoid crossed products} still holds. The proof is identical to the proof given here except that one has to use \cite[Thm. 4.1]{LalondeNuclearity} for the construction of tensor product bundles and \cite[Thm. 4.2]{LalondeNuclearity} instead of Lemma \ref{lem:maximalandmaximalproduct}.

 \bibliography{references}{}
	\bibliographystyle{alpha}
\end{document}